\documentclass[a4paper]{article}

\usepackage[english]{babel} 
\usepackage[utf8]{inputenc}

\usepackage[round]{natbib}

\usepackage{amsmath}
\usepackage{amsthm} 
\usepackage{amsfonts} 
\usepackage{bbm} 
\usepackage{mathtools}
\usepackage{enumerate}
\usepackage{xcolor} 
\usepackage{listings} 
\usepackage[algosection]{algorithm2e}
\RestyleAlgo{ruled}
\SetKwInput{Input}{Input}
\SetKwInput{Output}{Output}
\usepackage{graphicx}
\usepackage{hyperref}

\title{Optimal reinsurance in a competitive market}
\date{}
\newcommand{\footremember}[2]{%
    \footnote{#2}
    \newcounter{#1}
    \setcounter{#1}{\value{footnote}}%
}
\newcommand{\footrecall}[1]{%
    \footnotemark[\value{#1}]%
} 

\author{
Lea Enzi\footremember{TUG}{Institute of Statistics, Graz University of Technology, Kopernikusgasse 24/III, 8010 Graz, Austria.
lea.enzi@tugraz.at, stefan.thonhauser@tugraz.at} \& Stefan Thonhauser\footrecall{TUG}
  }
\date{}

\newtheorem{theorem}{Theorem}[section]

\newtheorem{lemma}{Lemma}[section]
\newtheorem{corollary}{Corollary}[section]
\theoremstyle{remark}
\newtheorem{remark}{Remark}[section]

\theoremstyle{definition}
\newtheorem{definition}{Definition}[section]

\numberwithin{equation}{section}

\newcommand*\diff{\mathop{}\!\mathrm{d}}

\newcommand{\ind}{\mathbbm{1}_}

\DeclareMathOperator*{\argmax}{arg\,max}
\DeclareMathOperator*{\argmin}{arg\,min}

\begin{document}

\maketitle
\section*{Abstract}
We study a stochastic differential game in a ruin theoretic environment. In our setting two insurers compete for market share, which is represented by a joint performance functional. Consequently, one of the insurers strives to maximize it, while the other seeks to minimize it. As a modelling basis we use classical surplus processes extended by dynamic reinsurance opportunities, which allows us to use techniques from the theory of piecewise deterministic Markov processes to analyze the resulting game. In this context we show that a dynamic programming principle for the upper and lower value of the game holds true and that these values are unique viscosity solutions to the associated Bellman-Isaacs equations.
Finally, we provide some numerical illustrations.

\section{Introduction}
Stochastic differential games put players in a dynamic competition where both strategy and chance influence the outcome. \cite{FlemingSouganadis1989} laid the groundwork for the study of stochastic differential games. They considered two-player zero-sum stochastic differential games and established the key principles regarding value functions and associated Bellman-Isaacs equations. Numerous publications were ultimately based on this. \cite{Ferreira2019} for example extended this theory to a random time horizon. \cite{BuckdahnLi2008} considered a problem where the cost functionals are given by backward stochastic differential equations. An overview on different approaches together with a ``weak formulation" of stochastic differential games was given by \cite{Possamai2020}.

This theory also proves to be suitable in the context of reinsurance, where insurers navigate a dynamic competition against each other and the uncertainties of claim arrivals. Modeling these interactions as a stochastic differential game provides a framework for investigating a range of problems. For example, \cite{TaksarZeng2011} considered a reinsurance game where each of the players' wealth is represented by a surplus process which can be controlled individually by acquiring reinsurance. \cite{Thogersen2019} on the other hand considered a game using the premium as control variate, accounting for the number of potential costumers.

A considerable branch of actuarial literature deals with so-called reinsurance games based on Stackelberg equilibria, see for instance \cite{Chen_Shen_2018} or \cite{CAO2022128}.
These games involve a leader and a follower with differing decision-making processes, resulting in an asymmetric situation.

The majority of existing literature relies on a diffusion approximation. However, we  use a different modeling framework, which captures the specific characteristics of our problem better.
We consider a similar game as in \cite{TaksarZeng2011}, where we use the ``strategy against control" approach, introduced in \cite{FlemingSouganadis1989}. The surplus process is given by the classical risk model, but we make use of the fact that this corresponds to a piecewise deterministic Markov process (PDMP). These types of processes were introduced by \cite{Davis1984,Davis}. We leverage the unique structure of these processes by exploiting the properties of the underlying jump filtration. A benefit of our approach is that we avoid the need for ``restrictive strategies", due to the  ``simpler" filtration compared to diffusion processes. A comprehensive theoretical treatment of jumping Markov processes can be found in \cite{Skorokhod1996}.

For the reinsurance control, we use the classical approach presented by \cite{Schmdili2001,Schmidli2008} who aims at minimizing the probability of ruin. Further works related to this topic are for example by \cite{AzMul2005,AzcueMuler2014} who consider a dividend maximization problem, or \cite{Eisenberg2011} who minimize capital injections into an insurance portfolio.

The paper is structured as follows. In Section 2, we introduce the specific underlying insurance model and the stochastic differential game. In Section 3, we study the characteristics of the value functions and show that the dynamic programming principle holds. In  Section 4, we prove that the value functions are viscosity solutions of the associated Bellman-Isaacs equations and that a comparison result holds true. In Section 5, we numerically solve these equations and compute approximations to the maximizing and minimizing controls.

\section{Model setup}
\subsection{The underlying insurance model} 
\label{subsec:Modelsetup}

Let $(\Omega, \mathcal{F},\mathbb{P})$ be a probability space. 
We consider two insurance portfolios with surplus processes of the form
\begin{equation}
\label{surplusProcesses}
	X_{t}^i = x_i + c_i t - \sum_{j=1}^{N^i_t}Y_j^i, \quad i = 1,2,
\end{equation}
where $x_i\geq 0$ is the initial capital, $c_i>0$ the premium rate and $N^i=(N^i_t)_{t\geq 0}$ is a homogeneous Poisson process with intensity $\lambda_i>0$. Furthermore, $\left\lbrace Y_j^i\right\rbrace _{j \in \mathbb{N}}$ is a sequence of i.i.d. random variables with continuous distribution function $F_Y^i$, where $F_Y^i(0)=0$. Moreover, $N^1, N^2, \{Y_j^1\}_{j \in \mathbb{N}}$ and $\{Y_j^2\}_{j \in \mathbb{N}}$ are mutually independent.

These processes are \textit{piecewise deterministic Markov processes} (PDMPs), which were introduced by \cite{Davis1984,Davis} and have a special structure that we will utilize later. Indeed, let $(\mathcal{Z},\mathcal{A},\ell)$ be the unit interval probability space, where $\mathcal{Z}=[0,1]$, $\mathcal{A}$ is the class of Lebesgue measurable sets and $\ell$ the Lebesgue measure. Then, the random variable $U:\mathcal{Z}\rightarrow \mathbb{R}$ where $U(z)=z$ is uniformly distributed on $\mathcal{Z}$. Now, let $\Omega=\prod_{i=1}^{\infty}\mathcal{Z}_i$, where $\{\mathcal{Z}_i\}_{i \in \mathbb{N}}$ are copies of $\mathcal{Z}$. $\Omega$ is called \textit{Hilbert cube} and together with the product $\sigma$-algebra $\mathcal{F}$ and product measure $\mathbb{P}$ it builds the sample space for a countable sequence of i.i.d. random variables $U_1,U_2,\dots$, where each $U_n\sim U(0,1)$. For each $\omega=(\omega_1,\omega_2,...) \in \Omega$, $U_n(\omega)=\omega_n$. Now, for fixed $i \in \{1,2\}$, consider such a sequence of uniform random variables $\{U_j\}_{j \in \mathbb{N}}$ and define
\begin{align}
\label{eq:definitionofRV}
    &S_n^i(\omega)=-\lambda_i^{-1}\ln(1-U_{2n-1}(\omega)), \quad T_n^i(\omega)=\sum_{k=1}^{n}S_k^i(\omega),\\
    &X_{T_n^i}^i(\omega)=X_{T_{n-1}^i}^i(\omega)+c_i S_n^i(\omega)-(F_Y^{i})^{-1}(U_{2n}(\omega)), \nonumber\\
    &N_t^i(\omega) = \sum_{k=1}^{\infty}\mathbbm{1}_{\{ T_k^i(\omega)\leq t\}}, \nonumber
\end{align}
for $ n \in \mathbb{N}$, where $T_0^i:=0$ and $X_0^i=x_i$. These definitions correspond to \eqref{surplusProcesses}. For $i \in \{1,2\}$ the generator of this PDMP is given by 
\begin{equation*}
    \mathcal{A}^i f(x) = c_i f'(x) + \lambda_i \int_{0}^\infty f(x-y) -f(x) \diff F_Y^i(y),
\end{equation*}
for $f \in \mathcal{D}(\mathcal{A}^i)$, which is the domain of the generator (see e.g. \cite{Davis} or \cite{Rolski}).

\subsection{A two players game}
We consider a game similar to the one analyzed by \cite{TaksarZeng2011}. We assign the surplus processes \eqref{surplusProcesses} to two insurance companies, which compete on a common insurance market. We refer to the company corresponding to $i=1$ as player 1 and the other as player 2. The players can observe each other's strategy. Therefore, it is necessary to employ a common filtration, i.e., the $\mathbb{P}$-completed filtration $\{\mathcal{F}_t\}_{t\geq 0}$ which is generated by the combined pure jump process. We define jump times similarly to \eqref{eq:definitionofRV}, i.e., for each $\omega \in \Omega$ and i.i.d. uniform random variables $U_1,U_2,\dots$ with $U_n(\omega)=\omega_n$,
\begin{equation}
    \label{eq:defJumptimes}
    S_n(\omega)=-(\lambda_1+\lambda_2)^{-1}\ln(1-U_{2n-1}(\omega)), \quad T_n(\omega)=\sum_{k=1}^{n}S_k(\omega).
\end{equation}
Now, we consider reinsurance controls of the following form:\\
Let $A_1, A_2 \subset \mathbb{R}$ be compact sets. Then, the reinsurance controls $u_1=(u_t^1)_{t\geq 0}$ and $u_2=(u_t^2)_{t\geq 0}$ are $A_1$ and $A_2$-valued stochastic processes respectively, which are predictable on $\Omega$ in the sense of Definition~(26.3) in \cite[p.~67]{Davis}. This means, that there exists a measurable function $g_1^i:\mathbb{R}_0^+\times\mathbb{R}\rightarrow A^i$ and measurable functions $g_k^i:\mathbb{R}_0^+ \times (\mathbb{R}_0^+\times \mathbb{R})^{k-1}\rightarrow A^i$, $k=1,2,\dots$, such that 
\begin{align}
\label{predictable}
    u^i_t(\omega) = &g_1^i(t,x)\mathbbm{1}_{\{t\leq T_1(\omega)\}} \nonumber\\
    &+\sum_{k=2}^{\infty}g_k^i\left(t, S_1(\omega), X_{T_1}(\omega),\dots,S_{k-1}(\omega),X_{T_{k-1}}(\omega)\right)\mathbbm{1}_{\{T_{k-1}(\omega)<t\leq T_{k}(\omega)\}},
\end{align} 
for $i=1,2$.
We call these controls \textit{admissible} and denote the set of all such by $\mathcal{U}_1$ and $\mathcal{U}_2$ for $i=1,2$ respectively.

The part of the claim the first insurer has to pay is given by a retention or risk share function $r: \mathbb{R}_0^+\times (A_1 \cup A_2) \rightarrow \mathbb{R}_0^+$. Thus, if a claim of size $y$ occurs at time $t$, the insurer has to pay $r(y,u_t^i)$ and $y-r(y,u_t^i)$ is covered by the reinsurer. Consequently, it is natural to assume that $r$ is monotone increasing in $y$ and $r(y,\cdot) \in [0,y]$. Further, we assume that $r(y,\cdot):(A_1 \cup A_2) \rightarrow \mathbb{R}_0^+$ is continuous for all $y \in \mathbb{R}_0^+$. The premium the first insurer has to pay for the reinsurance is described by some measurable and continuous function $p_i: A_i \rightarrow \mathbb{R}$, which reduces the premium incomes of the first insurer to
\begin{equation*}
	c_i(u_t^i)=c_i - p_i(u_t^i), \quad i=1,2.
\end{equation*}
If one assumes that reinsurance is expensive in the sense that a full cover results in a negative premium, then one would impose existence of a control $\hat{u}_i \in \mathcal{U}_i$ such that 
\begin{equation*}
	c_i(\hat{u}_i):= \pi_i <0, \quad i=1,2.
\end{equation*}
In general, we will not assume this feature. 
\begin{remark}
    The theoretical framework remains unaffected by considering distinct retention levels, $r_1$ and $r_2$, for the individual players. For simplicity of notation, we assume identical reinsurance types.
\end{remark}

The controlled surplus processes are now given by
\begin{equation*}
	X_t^{i, u_i}=x_i+\int_{0}^{t}c_i(u_s^i) \diff s- \sum_{j=1}^{N_t^i}r(Y_j^i,u_{T_j^i}^i), \quad i=1,2.
\end{equation*}
\noindent
We define a new controlled process $X^{u_1 u_2}$ as the difference of the two surplus processes
\begin{equation}
	X_t^{u_1u_2}=x+\int_{0}^{t}\left( c_1(u_s^1) - c_2(u_s^2) \right) \diff s - \sum_{j=1}^{N_t^1}r(Y_j^1,u_{T_j^1}^1)+\sum_{k=1}^{N_t^2}r(Y^2_k,u_{T_k^2}^2),
\end{equation}
 where $x:=x_1-x_2$. We denote the common jump process by $N=N^1+N^2$ with jump intensity $\lambda=\lambda_1+\lambda_2$. Thus, $N_t\sim$ Pois$(\lambda t)$ with jump times $\{T_j\}_{j \in \mathbb{N}}$. The definition of the jump times remains consistent with \eqref{eq:defJumptimes}.
Let $a,b$ be real numbers with $a<x<b$ and \begin{equation*}
     \tau^{u_1 u_2}_x=\inf\left\lbrace t>0: X_t^{u_1u_2} \notin [a,b] \vert X_0=x \right\rbrace.
\end{equation*}
We consider performance functionals of the form
\begin{align}
\label{functional}
	J^{u_1 u_2}(x)=\mathbb{E}_x  \left[  \int_{0}^{\tau^{u_1u_2}}e^{-\delta t}\zeta(X_t^{u_1u_2})\diff t + e^{-\delta\tau^{u_1u_2}} h(X_{\tau^{u_1u_2}}^{u_1u_2})  \right],
\end{align}
where $\delta>0$ and $\zeta: [a,b]\rightarrow\mathbb{R}_0^+$ is a measurable and Lipschitz continuous function. Further, $h:\mathbb{R} \rightarrow \mathbb{R}$ is measurable, Lipschitz continuous on $\mathbb{R} \backslash (a,b)$ and monotone increasing. Furthermore, it fulfills \begin{equation*}
    \int_{\mathbb{R}} h(y)Q(x,\diff y) =  \int_{\mathbb{R}_0^+}h(x-y) \diff F_Y^1(y) + \int_{\mathbb{R}_0^+}h(x+y) \diff F_Y^2(y) <\infty
\end{equation*} for all $x \in [a,b]$, where $Q$ denotes the jump kernel of the process $X$.\\
Player one wants to maximize, and player two wants to minimize the functional \eqref{functional}, using their respective controls $u_1 \in \mathcal{U}_1$ and $u_2 \in \mathcal{U}_2$. 

While in \cite{TaksarZeng2011}, the set of controls is immediately restricted to Markovian controls, we want to follow another - more general -  approach based on \cite{FlemingSouganadis1989}. 

\begin{definition}
\label{def:admissible strategy}
Let $u$ and $ \tilde{u}$ be two controls. We say that  $u \approx \tilde{u}$ if $u = \tilde{u}$ $\ell \times \mathbb{P}$-a.s. on compact time intervals. \\
    An admissible strategy for player 1 is a mapping $\alpha:\mathcal{U}_2 \rightarrow \mathcal{U}_1$ for which $u_2 \approx \tilde{u}_2$ implies $\alpha[u_2] \approx \alpha[\tilde{u}_2]$. The set of all such strategies is denoted by $\Gamma$.\\
    Equivalently, a mapping $\beta:\mathcal{U}_1 \rightarrow \mathcal{U}_2$ for which $u_1 \approx \tilde{u}_1 $ implies $\beta[u_1] \approx \beta[\tilde{u}_1]$, is called an admissible strategy for player 2. The set of all such strategies is denoted by $\Delta$.
    \end{definition}
With this strategy against control formulation, the lower value of the game is defined as
\begin{equation}
\label{lowergame}
     \underline{V}(x) = \inf_{\beta \in \Delta} \sup_{u_1 \in \mathcal{U}_1} J^{u_1\beta(u_1)}(x),
\end{equation}
and the game's upper value is
\begin{equation}
 \label{uppergame}  
     \overline{V}(x) = \sup_{\alpha \in \Gamma} \inf_{u_2 \in \mathcal{U}_2 } J^{\alpha(u_2) u_2}(x).
\end{equation}

\begin{remark}
    The precise definitions of the functions $h$ and $\zeta$ leave room for interpretation of the game. For example, $h$ correlates to what happens if the surplus of the first player is much larger than the one of the second player or vice versa. A possible interpretation is that if this event occurs, some restructuring of the ``weaker" player occurs and the game restarts, or the two companies merge etc. With $\zeta$, we can model for example running costs or gains but also provide a reason that player 1 might not want to exceed $b$ immediately. All in all, the decisions only depend on a function of the difference of the surplus processes which describes the (possible) dominance of one player over the other. In our setting the two insurers are linked to each other by the shared performance functional and control interactions.
\end{remark}

\section{A-priori properties of the value functions}
We start by showing that the performance functional $J^{u_1u_2}$ for two controls is quite regular. For $u_1 \in \mathcal{U}_1$ and  $u_2 \in \mathcal{U}_2$ which have the form \eqref{predictable}, we write $\phi^{u_1 u_2}(\cdot,\cdot)$ for the function given by
\begin{align*}
\phi^{u_1 u_2}(t,x)=x+\int_0^t(c_1(g^1_k(s,w_{k-1}))-c_2(g^2_k(s,w_{k-1})))\diff s,
\end{align*}
for $t\in[0,S_k)$, $x=X_{T_{k-1}}$ and $w_{k-1}=(S_1,X_{T_1},\ldots,S_{k-1},X_{T_{k-1}})$. The parameters $S_j$ and $X_{T_j}$ correspond to former inter-jump times and respective post-jump locations. The map $t\mapsto\phi^{u_1u_2}(t,x)$ is the deterministic path of the controlled process between to consecutive jump times.
\begin{lemma}
\label{Lemma:J bounded and continuous}
    The function $t\mapsto J^{u_1 u_2}(\phi^{u_1u_2}(t,x))$ is bounded and continuous for any controls $u_1 \in \mathcal{U}_1$ and $u_2 \in \mathcal{U}_2$.
\end{lemma}

\begin{proof}
    Let $x \in (a,b)$ and $u_1 \in \mathcal{U}_1, u_2 \in \mathcal{U}_2$.
    We start with the boundedness. Let $M_h := \sup_{x \in [a,b]} \int_{0}^{\infty}\vert h(y)\vert Q(x,\diff y)$. Then, 
    \begin{align*}
		\vert J^{u_1 u_2}(x) \vert		&\leq \Vert \zeta \Vert_{\infty} \mathbb{E}_x \left[  \int_{0}^{\tau^{u_1 u_2}} e^{-\delta t} \diff t\right]  + \mathbb{E}_x \left[ \vert h(X_{\tau^{u_1u_2}}^{u_1u_2}) \vert\right] \\
										&\leq \frac{\Vert \zeta \Vert_{\infty}}{\delta} + M_h =: M.
	\end{align*}
 For the continuity, we need the form of a predictable process \eqref{predictable}, which gives us the following: For some time $t \in \mathbb{R}_0^+$, there exist measurable functions $g^1$ and $g^2$ such that $u_t^1 = g^1(t,x)$ and $u_t^2=g^2(t,x)$ until the first jump occurs. Let $t^*$ be such that the process does not reach the boundary continuously before this time. Then,
 \begin{align*}
     J^{u_1u_2}(x)&= \mathbb{P}(T_1>t^*) \left(  \int_{0}^{t^*}  e^{-\delta t} \zeta\left(x+\int_{0}^{t}c_1(g^1(s,x))-c_2(g^2(s,x)) \diff s\right) \diff t  \right. \\
     &\qquad \left. + \mathbb{E}_x \left[ \int_{t^*}^{\tau^{u_1 u_2}} e^{-\delta t} \zeta(X_t^{u_1 u_2}) \diff t + e^{- \delta \tau^{u_1 u_2}} h(X_{\tau^{u_1 u_2}}^{u_1u_2}) \Bigg \vert T_1>t^* \right] \right) \\
     &+\mathbb{P}(T_1\leq t^*) \mathbb{E}_x \left[ \int_{0}^{\tau^{u_1 u_2}} e^{-\delta t} \zeta(X_t^{u_1 u_2}) \diff t + e^{- \delta \tau^{u_1 u_2}} h(X_{\tau^{u_1 u_2}}^{u_1u_2}) \Bigg \vert T_1\leq t^* \right].
 \end{align*}
 By using the tower and the Markov property, we get that
 \begin{align*}
     J^{u_1u_2}(x)&= \mathbb{P}(T_1>t^*) \left( \int_{0}^{t^*}  e^{-\delta t} \zeta\left(x+\int_{0}^{t}c_1(g^1(s,x))-c_2(g^2(s,x)) \diff s\right) \diff t  \right. \\
     &\qquad \left. + e^{-\delta t^*} J^{\tilde{u}_1 \tilde{u}_2}\left(x+ \int_{0}^{t^*}c_1(g^1(t,x))-c_2(g^2(t,x))\diff t \right) \right)\\
      &+\mathbb{P}(T_1\leq t^*) \mathbb{E}_x \left[ \int_{0}^{\tau^{u_1 u_2}} e^{-\delta t} \zeta(X_t^{u_1 u_2}) \diff t + e^{- \delta \tau^{u_1 u_2}} h(X_{\tau^{u_1 u_2}}^{u_1u_2}) \Bigg \vert T_1\leq t^* \right],
 \end{align*}
 where $\tilde{u}_t^i = u_{t^*+t}^i, i=1,2$. Thus,
 \begin{align*}
     \Bigg \vert J^{u_1u_2}(x) - J^{\tilde{u}_1 \tilde{u}_2}\left(x+ \int_{0}^{t^*}c_1(g^1(t,x))-c_2(g^2(t,x))\diff t \right) \Bigg \vert\\
     \leq \left\vert\mathbb{P}(T_1>t^*)e^{-\delta t^*}-1\right\vert M + \Vert\zeta \Vert_{\infty}t^*+ \mathbb{P}(T_1\leq t^*) M \\= M\left(2-e^{-(\lambda_1+\lambda_2)t^*} (e^{-\delta t^*}+1)\right) + \Vert\zeta \Vert_{\infty}t^*.
 \end{align*}
 For $t^* \rightarrow 0$ the r.h.s. converges to $0$. Since $u_i$ and $\tilde{u}_i$ correspond to the same function $g^i$ for $i=1,2$, we are done.
\end{proof}

\begin{corollary}
The above proof also shows that $J^{u_1u_2}$ is absolutely continuous along the ODE path, independent of $u_1$ and $u_2$. Therefore, it is in the domain of the generator of $X^{u_1 u_2}$. 
\end{corollary}

\begin{lemma}
$\underline{V}$ and $\overline{V}$ are bounded and continuous.
\end{lemma}

\begin{proof}
The boundedness follows directly from Lemma \ref{Lemma:J bounded and continuous}.
For the continuity, we start with $\underline{V}$. 
Let $x>y$ and consider the difference $\underline{V}(x)-\underline{V}(y)$, 
    \begin{align*}
         \underline{V}(x)-\underline{V}(y) =\inf_{\beta \in \Delta} \sup_{u_1 \in \mathcal{U}_1}J^{u_1 \beta(u_1)}(x)- \inf_{\beta \in \Delta} \sup_{u_1 \in \mathcal{U}_1}J^{u_1 \beta(u_1)}(y).
    \end{align*}
For some $\varepsilon>0$, let $\beta^y$ be an admissible $\varepsilon$-optimal strategy corresponding to $\underline{V}(y)$, and $u^x$ be an admissible $\varepsilon$-optimal control corresponding to $\underline{V}(x)$. We denote by  $X^x$ and $X^y$ the controlled processes starting in $x$ and $y$ respectively, with corresponding exit times $\tau_x$ and $\tau_y$. Notice, that $u^x$ is also an admissible control for the process starting in $y$. From Definition \ref{def:admissible strategy}, the answering strategy $\beta^y(u^x)$ is the same for both processes. Consequently, for a fixed $\omega$ the two processes move parallel until time $\tau_x \wedge \tau_y$. We get
    \begin{align*}
        &\underline{V}(x)-\underline{V}(y) \leq \mathbb{E}\left[\int_{0}^{\tau_x \wedge \tau_y} e^{-\delta t}(\zeta(X_t^x)-\zeta(X_t^y)) \diff t\right] \\
        & \quad+ \mathbb{E}\left[\int_{\tau_y}^{\tau_x}e^{-\delta t}\zeta(X_t^x)\diff t \ \ind{\{\tau_y<\tau_x\}}\right]
        - \mathbb{E}\left[\int_{\tau_x}^{\tau_y}e^{-\delta t}\zeta(X_t^y)\diff t \ \ind{\{\tau_x<\tau_y\}}\right]\\
        & \quad+ \mathbb{E}\left[ e^{-\tau_x}h(X_{\tau_x}^x)-e^{-\tau_y}h(X_{\tau_y}^y)\right] + 2 \varepsilon\\
        & \quad \leq \int_0^{\infty}e^{-\delta t} L_\zeta (x-y) \diff t +L_h|x-y|
        + 2\,M\left(\mathbb{P}(\tau_y < \tau_x) + \mathbb{P}(\tau_x < \tau_y)\right) + 2 \varepsilon.
    \end{align*}
    Let $T_x:=\inf\{t \geq 0: X_t^x \leq a+(x-y) \vee X_t^x \geq b+(x-y)$\}. 
    Then, $\mathbb{P}(\tau_y < \tau_x) = \mathbb{P}(T_x<\tau_x)$,
    which goes to zero if $y \rightarrow x$.
    With a similar argument, $\mathbb{P}(\tau_x < \tau_y)$ goes to zero too.\\
    The same holds for $\underline{V}(y)-\underline{V}(x)$ by choosing suitable controls and strategies. Therefore, $|\underline{V}(x)-\underline{V}(y)|\rightarrow 0$ as $|x-y|\rightarrow 0$.\\
    Now, we consider
    \begin{equation*}
        \overline{V}(x)-\overline{V}(y) =\sup_{\alpha \in \Gamma} \inf_{u_2 \in U_2}J^{\alpha(u_2) u_2}(x)- \sup_{\alpha \in \Gamma} \inf_{u_2 \in U_2}J^{\alpha(u_2) u_2}(y)
    \end{equation*}
By choosing an admissible $\varepsilon$-optimal strategy $\alpha^x$ corresponding to $\overline{V}(x)$ and an admissible $\varepsilon$-optimal control $u^y$ corresponding to $\overline{V}(y)$ we get by similar arguments as before, that $|\overline{V}(x)-\overline{V}(y)|\rightarrow 0$ as $|x-y|\rightarrow 0$. 
\end{proof}

\begin{remark}
One may notice that in the following theorem we do not need to deal with so-called \emph{r-strategies} as elaborated in \cite{FlemingSouganadis1989}. This is because of the natural form of predictable processes with respect to the pure jump filtration, see \eqref{predictable}.
\end{remark}
\newpage
\begin{theorem}[Dynamic programming principle]
\label{thm:DPP}
 Under our model assumptions, the following holds for any $\{\mathcal {F}_t\}_{t\geq 0}$ bounded stopping time $T$:
  \begin{align}
        \label{DPPlower}
        \underline{V}(x)= &\inf_{\beta \in \Delta} \sup_{u_1 \in \mathcal{U}_1} \mathbb{E}_x \Biggl[ \int_{0}^{\tau^{u_1\beta(u_1)}\wedge T}e^{-\delta s} \zeta\left(X_s^{u_1\beta(u_1)}\right) \diff s \\
        &+ e^{-\delta T} \underline{V}\left(X^{u_1\beta(u_1)}_{T}\right) \ind{\{T<\tau^{u_1\beta(u_1)}\}}+ e^{-\delta \tau^{u_1\beta(u_1)}}h(X_{\tau^{u_1\beta(u_1)}}^{u_1\beta(u_1)}) \ind{\{T\geq\tau^{u_1\beta(u_1)}\}} \Biggr],  \nonumber\\
        \overline{V}(x)= &\sup_{\alpha \in \Gamma} \inf_{u_2 \in \mathcal{U}_2 } \mathbb{E}_x \Biggl[ \int_{0}^{\tau^{\alpha(u_2)u_2}\wedge T}e^{-\delta s} \zeta\left(X_s^{\alpha(u_2)u_2}\right) \diff s \label{DPPupper}\\ &+ e^{-\delta T} \overline{V}\left(X^{\alpha(u_2)u_2}_{T}\right) \ind{\{T<\tau^{\alpha(u_2)u_2}\}}  +  e^{-\delta \tau^{\alpha(u_2)u_2}} h(X^{\alpha(u_2)u_2}_{\tau^{\alpha(u_2)u_2}}) \ind{\{T\geq\tau^{\alpha(u_2)u_2}\}} \Biggr].\nonumber  
    \end{align}
\end{theorem}

\begin{proof}
We only prove \eqref{DPPlower}, since \eqref{DPPupper} follows similarly.\\
Let $T\geq 0$. We define 
\begin{align*}
    W(x):= &\inf_{\beta \in \Delta} \sup_{u_1 \in \mathcal{U}_1} \mathbb{E}_x \Biggl[ \int_{0}^{\tau^{u_1\beta(u_1)}\wedge T}e^{-\delta s} \zeta\left(X_s^{u_1\beta(u_1)}\right) \diff s \\ &+ e^{-\delta T} \underline{V}\left(X^{u_1\beta(u_1)}_{T}\right) \ind{\{T<\tau^{u_1\beta(u_1)}\}} \nonumber+ e^{-\delta \tau^{u_1\beta(u_1)}}h(X_{\tau^{u_1\beta(u_1)}}^{u_1\beta(u_1)}) \ind{\{T\geq\tau^{u_1\beta(u_1)}\}} \Biggr]. 
\end{align*}
Let $\varepsilon>0$. There exists a strategy $\beta_\varepsilon \in \Delta$ such that
\begin{multline}
    \label{eq:Wineq}
    W(x)\geq  \mathbb{E}_x \Biggl[ \int_{0}^{\tau^{u_1\beta_\varepsilon(u_1)}\wedge T}e^{-\delta s} \zeta\left(X_s^{u_1\beta_\varepsilon(u_1)}\right) \diff s + e^{-\delta T} \underline{V}\left(X^{u_1\beta_\varepsilon(u_1)}_{T}\right) \ind{\{T<\tau^{u_1\beta_\varepsilon(u_1)}\}} \\+ e^{-\delta \tau^{u_1\beta_\varepsilon(u_1)}}h(X_{\tau^{u_1\beta_\varepsilon(u_1)}}^{u_1\beta_\varepsilon(u_1)}) \ind{\{T\geq\tau^{u_1\beta_\varepsilon(u_1)}\}} \Biggr] - \varepsilon,
\end{multline}
for all $u_1 \in \mathcal{U}_1$.\\
Furthermore, for each $\hat{x} \in [a,b] $ there exists a $\beta_{\hat{x}} \in \Delta $ s.t.,
\begin{equation}
\label{eq:Vepsilonstrat}
     \underline{V}(\hat{x})>\sup_{u_1 \in \mathcal{U}_1}J^{u_1 \beta_{\hat{x}}(u_1)}(\hat{x})-\varepsilon.
\end{equation}
Since $\underline{V}$ is continuous and for fixed $u_1$ and $\beta$, $J^{u_1 \beta(u_1)}$ is continuous too (for fixed $t\mapsto u_1(t)=(u^1_t,u^2_t)$), we can choose a grid $a=x_0<x_1<\dots<x_{n-1}<x_n=b$, s.t. for $u_1 \in \mathcal{U}_1$ and $\beta \in \Delta$,
\begin{equation}
\label{eq:Jineq}
    \vert J^{u_1 \beta(u_1)}(x_{i})-J^{u_1 \beta(u_1)} (x_{i-1})\vert<\varepsilon,
\end{equation}
and
\begin{equation}
\label{eq:Vineq}
   |\underline{V}(x_{i})-\underline{V}(x_{i-1})|<\varepsilon,
\end{equation} for $i \in \{1,\dots,n\}$. \\
For $u_1 \in \mathcal{U}_1$ and $s\geq 0$, we define
\begin{equation}
\label{eq:newbeta}
    \beta(u_1)(\omega)_s = \begin{cases} \beta_\varepsilon(u_1)(\omega)_s & s\leq T,\\
    \sum_{i=1}^n \mathbbm{1}_{[x_{i-1},x_i)}(X_T ^{u_1 \beta_\varepsilon(u_1)}) \beta_{x_{i-1}}(\tilde{u}_1)(\tilde{\omega})_{s-T} & s> T,
    \end{cases}
\end{equation}
where $\tilde{u}_1(t)=u_1(T+t)$, $\omega=(U_1(\omega),U_2(\omega),\dots)$ and $\tilde{\omega} = (U_{2m-1}(\omega),U_{2m}(\omega),\dots)$ s.t. $T_{m-1}(\omega)\leq T<T_m(\omega)$, according to \eqref{eq:defJumptimes}. 
From \eqref{eq:newbeta} we get:
\begin{align*}
    &J^{u_1 \beta(u_1)} (x)  =\mathbb{E}_x  \left[  \int_{0}^{\tau^{u_1\beta(u_1)}}e^{-\delta s}\zeta(X_s^{u_1\beta(u_1)})\diff s + e^{-\delta\tau^{u_1\beta(u_1)}} h(X_{\tau^{u_1\beta(u_1)}}^{u_1\beta(u_1)})  \right] \\
                           &  = \mathbb{E}_x  \Biggl[  \int_{0}^{\tau^{u_1\beta(u_1)}\wedge T}e^{-\delta s}\zeta(X_s^{u_1\beta(u_1)})\diff s  + e^{-\delta \tau^{u_1\beta(u_1)}}h(X_{\tau^{u_1\beta(u_1)}}^{u_1\beta(u_1)}) \\&  +\sum_{i=1}^n \mathbbm{1}_{[x_{i-1},x_i)}(X_T ^{u_1 \beta_\varepsilon(u_1)}) e^{-\delta T} \left( \int_{T}^{\tau^{u_1\beta(u_1)}}e^{-\delta (s-T)}\zeta(X_s^{u_1\beta(u_1)})\diff s  \right) \ind{\{T<\tau^{u_1\beta(u_1)}\}}
                            \Biggr].
\end{align*}
Using the tower property of conditional expectations yields,
\begin{align*}
                            &J^{u_1 \beta(u_1)} (x)  = \mathbb{E}_x  \Bigg[  \int_{0}^{\tau^{u_1\beta_\varepsilon(u_1)}\wedge T}e^{-\delta s}\zeta(X_s^{u_1\beta_\varepsilon(u_1)})\diff s \\
                            & \quad+ e^{-\delta \tau^{u_1\beta_\varepsilon(u_1)}}h(X_{\tau^{u_1\beta_\varepsilon(u_1)}}^{u_1\beta_\varepsilon(u_1)}) \ind{\{T\geq\tau^{u_1\beta_\varepsilon(u_1)}\}}\\
                            & \quad+ e^{-\delta T} \left.\sum_{i=1}^n \mathbbm{1}_{[x_{i-1},x_i)}(X_{T}^{u_1\beta_\varepsilon(u_1)})  \mathbb{E}_x\left[\left( \int_{T}^{\tau^{u_1\beta(u_1)}}e^{-\delta( s-T)}\zeta(X_s^{u_1\beta(u_1)})\diff s  \right. \right.\right.\\
                            & \quad \left. \left. + e^{-\delta (\tau^{u_1\beta(u_1)}- T)} h\left(X_{\tau^{u_1\beta(u_1)}}^{u_1 \beta(u_1)}\right) \right)\vert \mathcal{F}_T\right] \ind{\{T<\tau^{u_1 \beta_\varepsilon(u_1)}\}}\Bigg]\\
                            &=\mathbb{E}_x  \left[  \int_{0}^{\tau^{u_1\beta_\varepsilon(u_1)}\wedge T}e^{-\delta s}\zeta(X_s^{u_1\beta_\varepsilon(u_1)})\diff s \right. + e^{-\delta \tau^{u_1\beta_\varepsilon(u_1)}}h(X_{\tau^{u_1\beta_\varepsilon(u_1)}}^{u_1\beta_\varepsilon(u_1)}) \ind{\{T\geq\tau^{u_1\beta_\varepsilon(u_1)}\}}\\
                            &\quad +\left.\sum_{i=1}^n \mathbbm{1}_{[x_{i-1},x_i)}(X_{T}^{u_1\beta_\varepsilon(u_1)}) e^{-\delta T} J^{\tilde{u}_1\beta_{x_{i-1}}(\tilde{u}_1)}(X_{T}^{u_1\beta_\varepsilon(u_1)}) \ind{\{T<\tau^{u_1 \beta_\varepsilon(u_1)}\}} \right].
\end{align*}
By \eqref{eq:Vepsilonstrat} and \eqref{eq:Jineq}, we get that for $X_{T}^{u_1\beta_\varepsilon(u_1)}\in [x_i,x_{i+1})$,
\begin {align*}
    \underline{V}(x_i)\geq J^{\tilde{u}_1 \beta_{x_i}(\tilde{u}_1)}(x_i)-\varepsilon \geq J^{\tilde{u}_1 \beta_{x_i}(\tilde{u}_1)}(X_{T}^{u_1\beta_\varepsilon(u_1)})-2 \varepsilon.
\end{align*}
Note that this holds, because given the definition and purpose of $\tilde{u}_1$, the same control (specified for x and its path) is applied in both $J$ functionals above.
Therefore, and by using $\eqref{eq:Vineq}$,
\begin{align*}
    &J^{u_1 \beta(u_1)} (x) \\&\leq \mathbb{E}_x  \left[  \int_{0}^{\tau^{u_1\beta_\varepsilon(u_1)}\wedge T}e^{-\delta s}\zeta(X_s^{u_1\beta_\varepsilon(u_1)})\diff s \right. + e^{-\delta \tau^{u_1\beta_\varepsilon(u_1)}}h(X_{\tau^{u_1\beta_\varepsilon(u_1)}}^{u_1\beta_\varepsilon(u_1)}) \ind{\{T\geq\tau^{u_1\beta_\varepsilon(u_1)}\}}\\
                            &+\left.\sum_{i=1}^n \mathbbm{1}_{[x_{i-1},x_i)}(X_{T}^{u_1\beta_\varepsilon(u_1)}) e^{-\delta T} \underline{V}(x_{i-1}) \ind{\{T<\tau^{u_1 \beta_\varepsilon(u_1)}\}} \right]+2\varepsilon\\
                            & \leq \mathbb{E}_x  \left[  \int_{0}^{\tau^{u_1\beta_\varepsilon(u_1)}\wedge T}e^{-\delta s}\zeta(X_s^{u_1\beta_\varepsilon(u_1)})\diff s \right. + e^{-\delta \tau^{u_1\beta_\varepsilon(u_1)}}h(X_{\tau^{u_1\beta_\varepsilon(u_1)}}^{u_1\beta_\varepsilon(u_1)}) \ind{\{T\geq\tau^{u_1\beta_\varepsilon(u_1)}\}}\\
                            &+\left.\sum_{i=1}^n \mathbbm{1}_{[x_{i-1},x_i)}(X_{T}^{u_1\beta_\varepsilon(u_1)}) e^{-\delta T} \underline{V}(X_{T}^{u_1\beta_\varepsilon(u_1)}) \ind{\{T<\tau^{u_1 \beta_\varepsilon(u_1)}\}} \right]+3\varepsilon.
\end{align*}
Together with \eqref{eq:Wineq},
\begin{equation*}
    J^{u_1 \beta(u_1)} (x) \leq W(x)+4 \varepsilon,
\end{equation*}
and finally
\begin{equation}
\label{eq:v<=W}
    \underline{V}(x)\leq W(x)+4 \varepsilon.
\end{equation}
\ \\
For the other direction, we proceed in a similar way. We use that there exists some $\hat{\beta}_\varepsilon \in \Delta$ such that
\begin{equation}
\label{eq:V>}
    \underline{V}(x) \geq \sup_{u_1 \in \mathcal{U}_1} \mathbb{E}_x  \left[  \int_{0}^{\tau^{u_1\hat{\beta}_\varepsilon(u_1)}}e^{-\delta t}\zeta(X_t^{u_1\hat{\beta}_\varepsilon(u_1)})\diff t + e^{-\delta\tau^{u_1\hat{\beta}_\varepsilon(u_1)}} h(X_{\tau^{u_1\hat{\beta}_\varepsilon(u_1)}}^{u_1\hat{\beta}_\varepsilon(u_1)})  \right] - \varepsilon
\end{equation}
For such a fixed $\hat{\beta}_\varepsilon \in \Delta$, there exists $u_\varepsilon \in \mathcal{U}_1$ such that
\begin{align}
\label{eq:W<}
    W(x)\leq \mathbb{E}_x \Biggl[ \int_{0}^{\tau^{u_\varepsilon\hat{\beta}_\varepsilon(u_\varepsilon)}\wedge T}e^{-\delta s} \zeta\left(X_s^{u_\varepsilon\hat{\beta}_\varepsilon(u_\varepsilon)}\right) \diff s + e^{-\delta  T} \underline{V}\left(X^{u_\varepsilon\hat{\beta}_\varepsilon(u_\varepsilon)}_{T}\right)\ind{\{T<\tau^{u_1\hat{\beta}_\varepsilon(u_1)}\}} \nonumber \\+ e^{-\delta \tau^{u_1\hat{\beta}_\varepsilon(u_1)}}h(X_{\tau^{u_1\hat{\beta}_\varepsilon(u_1)}}^{u_1\hat{\beta}_\varepsilon(u_1)}) \ind{\{T\geq\tau^{u_1\hat{\beta}_\varepsilon(u_1)}\}} \Biggr] + \varepsilon.
\end{align}
Also, for each $\hat{x} \in [a,b]$, it holds that
\begin{align*}
    \underline{V}(\hat{x})\leq \sup_{u_1 \in \mathcal{U}_1}\mathbb{E}_{\hat{x}}\Biggl[ \int_{0}^{\tau^{u_1 \hat{\beta}_\varepsilon(u_1)}}e^{-\delta t} \zeta(X_t^{u_1 \hat{\beta}_\varepsilon(u_1)})\diff t + e^{-\delta \tau^{u_1 \hat{\beta}_\varepsilon(u_1)}} h\left(X_{\tau^{u_1 \hat{\beta}_\varepsilon (u_1)}}^{u_1 \hat{\beta}_\varepsilon(u_1)} \right)\Biggr],
\end{align*}
and therefore there exists $u_{\hat{x}} \in \mathcal{U}_1$ such that
\begin{equation}
\label{eq:V<J}
    \underline{V}(\hat{x})\leq J^{u_{\hat{x}} \hat{\beta}_\varepsilon(u_{\hat{x}})}(\hat{x})+\varepsilon.
\end{equation}
We define $u_1 \in \mathcal{U}_1$ by using the same grid as before, i.e., for $a=x_0<x_1<\dots<x_{n-1}<x_n=b$:

\begin{equation*}
    u_t^1 = 
    \begin{cases} 
    u_t^\varepsilon & t\leq T,\\
    \sum_{i=1}^n \mathbbm{1}_{[x_{i-1},x_i)}(X_T^{u_\varepsilon \hat{\beta}_\varepsilon(u_\varepsilon)}) \ u^{x_{i-1}}_{t-T} & t> T.
    \end{cases}
\end{equation*}
For $X_T^{u_\varepsilon \hat{\beta}_\varepsilon(u_\varepsilon)} \in [x_i,x_{i+1})$, we get by \eqref{eq:Jineq}, \eqref{eq:V<J} and \eqref{eq:Vineq},
\begin{equation}
\label{eq:JgeqJ}
    J^{u_{x_{i}}\hat{\beta}_\varepsilon(u_{x_{i}})}(X_T^{u_\varepsilon \hat{\beta}_\varepsilon(u_\varepsilon)})\geq J^{u_{x_{i}}\hat{\beta}_\varepsilon(u_{x_{i}})}(x_i)-\varepsilon\geq \underline{V}(x_i)-2 \varepsilon \geq \underline{V}(X_T^{u_\varepsilon \hat{\beta}_\varepsilon(u_\varepsilon)})-3\varepsilon.
\end{equation}
By similar computations as before, together with \eqref{eq:JgeqJ} and \eqref{eq:W<},
\begin{align*}
    &J^{u_1 \hat{\beta}_\varepsilon(u_1)}(x)\\ &= \mathbb{E}_x  \left[  \int_{0}^{\tau^{u_\varepsilon\hat{\beta}_\varepsilon(u_\varepsilon)}\wedge T}e^{-\delta s}\zeta(X_s^{u_\varepsilon\hat{\beta}_\varepsilon(u_\varepsilon)})\diff s \right. + e^{-\delta \tau^{u_\varepsilon\hat{\beta}_\varepsilon(u_\varepsilon)}}h(X_{\tau^{u_\varepsilon\hat{\beta}_\varepsilon(u_\varepsilon)}}^{u_\varepsilon\hat{\beta}_\varepsilon(u_\varepsilon)}) \ind{\{T\geq\tau^{u_\varepsilon\hat{\beta}_\varepsilon(u_\varepsilon)}\}}\\
                            & \qquad +\left.\sum_{i=1}^n \mathbbm{1}_{[x_{i-1},x_i)}(X_{T}^{u_\varepsilon\hat{\beta}_\varepsilon(u_\varepsilon)}) e^{-\delta T} J^{u_{x_{i-1}}\hat{\beta}_\varepsilon(u_{x_{i-1}})}(X_{T}^{u_\varepsilon\hat{\beta}_\varepsilon(u_\varepsilon)}) \ind{\{T<\tau^{u_\varepsilon \hat{\beta}_\varepsilon(u_\varepsilon)}\}} \right]\\
                            &\geq \mathbb{E}_x  \left[  \int_{0}^{\tau^{u_\varepsilon\hat{\beta}_\varepsilon(u_\varepsilon)}\wedge T}e^{-\delta s}\zeta(X_s^{u_\varepsilon\hat{\beta}_\varepsilon(u_\varepsilon)})\diff s \right. + e^{-\delta \tau^{u_\varepsilon\hat{\beta}_\varepsilon(u_\varepsilon)}}h(X_{\tau^{u_\varepsilon\hat{\beta}_\varepsilon(u_\varepsilon)}}^{u_\varepsilon\hat{\beta}_\varepsilon(u_\varepsilon)}) \ind{\{T\geq\tau^{u_\varepsilon\hat{\beta}_\varepsilon(u_\varepsilon)}\}}\\
                            & \qquad +\left.\sum_{i=1}^n \mathbbm{1}_{[x_{i-1},x_i)}(X_{T}^{u_\varepsilon\hat{\beta}_\varepsilon(u_\varepsilon)}) e^{-\delta T} \underline{V}(X_{T}^{u_\varepsilon\hat{\beta}_\varepsilon(u_\varepsilon)}) \ind{\{T<\tau^{u_\varepsilon \hat{\beta}_\varepsilon(u_\varepsilon)}\}} \right]-3 \varepsilon\\
                            &\geq W(x)-4\varepsilon.
\end{align*}
Finally, by \eqref{eq:V>},
\begin{equation}
\label{eq:W<=V}
    W(x)\leq \underline{V}(x)+5 \varepsilon.
\end{equation}
Hence, by \eqref{eq:v<=W} and \eqref{eq:W<=V}, 
\begin{equation*}
    \underline{V}(x)=W(x).
\end{equation*}
From this result for a deterministic $T$ one can derive the statement for bounded stopping times, using the particular form of $\{\mathcal{F}_t\}_{t\geq 0}$ stopping times \cite[p. 261]{Davis} and the standard procedure. 
\end{proof}

\section{Characterization of the value functions}
Since for the considered problem is not likely to allow for an explicit solution, we will have to rely on a numerical procedure. That is why we focus on existence and uniqueness of solutions to particular integro-differential equations.

\subsection{Viscosity solutions to Bellmann-Isaacs equations}
In this section we will show that $\underline{V}$ and $\overline{V}$ are viscosity solutions to their associated Bellman-Isaacs equations. 
For this purpose we define for $v$ a smooth function, $u_1 \in A_1$ and $u_2 \in A_2$, the following:
\begin{multline*}
\mathcal{L}(x,v(x),v'(x),v,u_1,u_2):= \left( c_1(u_1)-c_2(u_2)\right)  v'(x)\\
+ \lambda_1 \int_{0}^{\rho(x-a,u_1)}v(x-r(y,u_1)) \diff F_{Y}^1(y)+ \lambda_2 \int_{0}^{\rho(b-x,u_2)}v(x+r(y,u_2)) \diff F_Y^2(y)\\
+ \lambda_1 \int_{\rho(x-a,u_1)}^{\infty} h(x-r(y,u_1)) \diff F_Y^1(y) + \lambda_2 \int_{\rho(b-x,u_2)}^{\infty} h(x+r(y,u_2)) \diff F_Y^2(y)\\
-(\delta+\lambda_1 + \lambda_2) v(x) +\zeta(x),\;x \in [a,b],
\end{multline*}
where $\rho(z,u):= \inf\left\lbrace y \in \mathbb{R}_0^+: r(y,u)\geq z\right\rbrace$.
\begin{theorem}
\label{thm:Viscosity}
    \begin{enumerate}[(i)]
        \item $\underline{V}$ is a viscosity solution to the lower Bellman-Isaacs equation
        \begin{equation}\label{eq:BI_lower}
            \sup_{u_1 \in A_1}\inf_{u_2 \in A_2} \mathcal{L}(x,v(x),v'(x),v,u_1,u_2)=0,
        \end{equation}
        for $x \in (a,b)$.
        \item $\overline{V}$ is a viscosity solution to the upper Bellman-Isaacs equation
        \begin{equation}\label{eq:BI_upper}
            \inf_{u_2 \in A_2}\sup_{u_1 \in A_1} \mathcal{L}(x,v(x),v'(x),v,u_1,u_2)=0.
        \end{equation}
        for $x \in (a,b)$.
    \end{enumerate}
  \end{theorem}
\begin{proof}
    We only prove $(ii)$, since $(i)$ follows in a similar way. \\ We start by showing that $\overline{V}$ is a supersolution.
    Let $x \in (a,b)$ and let $\varphi$ be a smooth function on $[a,b]$, such that $\varphi(x)=\overline{V}(x)$ and $\overline{V}-\varphi\geq 0$. We have to show that  
    \begin{equation}
    \label{eq:supersolution}
        \inf_{u_2 \in A_2}\sup_{u_1 \in A_1}  \mathcal{L}(x,\varphi(x),\varphi'(x),\varphi,u_1,u_2) \leq 0.
    \end{equation}
    To do so, we assume that there exists some $\theta>0$ such that
    \begin{equation}
    \label{eq:L>0}
         \inf_{u_2 \in A_2}\sup_{u_1 \in A_1}  \mathcal{L}(x,\varphi(x),\varphi'(x),\varphi,u_1,u_2) \geq \theta > 0.
    \end{equation}
    We define,
    \begin{equation*}
        L(x,u_1,u_2):=  \mathcal{L}(x,\varphi(x),\varphi'(x),\varphi,u_1,u_2).
    \end{equation*}
    By \eqref{eq:L>0}, we have that
    \begin{equation*}
         \inf_{u_2 \in A_2}\sup_{u_1 \in A_1} L(x,u_1,u_2) \geq \theta.
    \end{equation*}
    Thus, for each $u_2 \in A_2$, there exists an $u_1\in A_1$ depending on $u_2$, such that 
    \begin{equation*}
        L(x,u_1,u_2)\geq \theta.
    \end{equation*}
    Note that $L$ is uniformly continuous in $u_1$ and $u_2$, since $c_1,c_2$ and $r(y,\cdot)$ are continuous functions on a compact set. Therefore,
    \begin{equation*}
        L(x,u_1,\xi) \geq \frac{3\theta}{4},
    \end{equation*}
    for all $\xi \in B(u_2,r) \cap A_2$ and some $r=r(u_2)>0$. \\
    Because $A_2$ is compact, there exist finitely many distinct points $a_1,\dots,a_n \in A_1$, $b_1,\dots,b_n \in A_2$ and $r_1,\dots,r_n>0$ such that
    \begin{equation*}
        A_2 \subset \bigcup_{i=1}^n B(b_i,r_i)
    \end{equation*}
    and
    \begin{equation}
        L(x,a_i,\xi)\geq \frac{3 \theta}{4}
    \end{equation}
    for $\xi \in B(b_i,r_i)\cap A_2$.\\
    We define a function $f:A_2 \rightarrow A_1$ by setting
    \begin{equation*}
        f(b) = \sum_{k=1}^n a_k \mathbbm{1}_{\{b \in B(b_k, r_k)\setminus\bigcup_{i=1}^{k-1}B(b_i,r_i)\}}.
    \end{equation*}
    Defining
    \begin{equation*}
        \alpha^*(u_2)(\omega)_t = f (u_2(\omega)_t),
    \end{equation*}
    leads to an admissible strategy $\alpha^* \in \Gamma$.\\
    Since $L$ is also uniformly continuous in $x$, there exists some $h>0$ such that for $\hat{x}$ with $\vert x-\hat{x}\vert<h$,
      \begin{equation*}
        L(\hat{x},f(u_2),u_2)\geq \frac{\theta}{2},
    \end{equation*}
    for all $u_2 \in A_2$.\\
    Furthermore, there exists some $t^*>0$ such that $\sup_{u_1 \in A_1,u_2\in A_2} \vert (c_1(u_1)-c_2(u_2))\vert t^*  \leq h$ and $x \pm \sup_{u_1 \in A_1,u_2\in A_2} \vert (c_1(u_1)-c_2(u_2))\vert t^* \in (a,b)$. We set $T:=T_1\wedge t^*$ and $\varphi(x)=0$ for $x \notin [a,b]$.
 
    By definition of $\varphi$ and Theorem \ref{thm:DPP}, it holds that
    \begin{align*}
        \varphi(x) &= \overline{V}(x) \\ &= \sup_{\alpha \in \Gamma} \inf_{u_2 \in \mathcal{U}_2 } \mathbb{E}_x \Biggl[ \int_{0}^{\tau^{\alpha(u_2)u_2}\wedge T}e^{-\delta s} \zeta\left(X_s^{\alpha(u_2)u_2}\right) \diff s + e^{-\delta T} \overline{V}\left(X^{\alpha(u_2)u_2}_{T}\right) \ind{\{T<\tau^{\alpha(u_2)u_2}\}} \nonumber \\ & \qquad +  e^{-\delta \tau^{\alpha(u_2)u_2}} h(X^{\alpha(u_2)u_2}_{\tau^{\alpha(u_2)u_2}}) \ind{\{T\geq\tau^{\alpha(u_2)u_2}\}} \Biggr] \\
        &\geq  \inf_{u_2 \in \mathcal{U}_2 } \mathbb{E}_x \Biggl[ \int_{0}^{\tau^{\alpha^*(u_2)u_2}\wedge T}e^{-\delta s} \zeta\left(X_s^{\alpha^*(u_2)u_2}\right) \diff s + e^{-\delta T} \varphi\left(X^{\alpha^*(u_2)u_2}_{T}\right) \nonumber \\ & \qquad +  e^{-\delta \tau^{\alpha^*(u_2)u_2}} h(X^{\alpha^*(u_2)u_2}_{\tau^{\alpha^*(u_2)u_2}}) \ind{\{T\geq\tau^{\alpha^*(u_2)u_2}\}} \Biggr].
    \end{align*}
     For some $\varepsilon>0$, there exists $u_2^\varepsilon \in \mathcal{U}_2$ such that 
    \begin{align*}
        \varphi(x) &\geq  \mathbb{E}_x \Biggl[ \int_{0}^{\tau^{\alpha^*(u_2^\varepsilon)u_2^\varepsilon}\wedge T}e^{-\delta s} \zeta\left(X_s^{\alpha^*(u_2^\varepsilon)u_2^\varepsilon}\right) \diff s + e^{-\delta T} \varphi\left(X^{\alpha^*(u_2^\varepsilon)u_2^\varepsilon}_{T}\right)  \nonumber \\ & \qquad +  e^{-\delta \tau^{\alpha^*(u_2^\varepsilon)u_2^\varepsilon}} h(X^{\alpha^*(u_2^\varepsilon)u_2^\varepsilon}_{\tau^{\alpha^*(u_2^\varepsilon)u_2^\varepsilon}}) \ind{\{T\geq\tau^{\alpha^*(u_2^\varepsilon)u_2^\varepsilon}\}} \Biggr] - \varepsilon.
    \end{align*}
   Then, the Dynkin formula yields
    \begin{align*}
        &\mathbb{E}_x\left[e^{-\delta T} \varphi(X_T^{\alpha^*(u_2^\varepsilon)u_2^\varepsilon})\right] \\&\qquad =\varphi(x)+\mathbb{E}_x\Bigg[\int_{0}^T e^{-\delta s} \mathcal{A}\varphi(X_s^{\alpha^*(u_2^\varepsilon)u_2^\varepsilon})-\delta e^{-\delta s} \varphi(X_s^{\alpha^*(u_2^\varepsilon)u_2^\varepsilon}) \diff s\Bigg].
    \end{align*}
     Hence, we get that
    \begin{align}
    \label{eq:Viscosityproof}
        \varepsilon &\geq \mathbb{E}_x \Biggl[ \int_{0}^{\tau^{\alpha^*(u_2^\varepsilon)u_2^\varepsilon}\wedge T}e^{-\delta s} \zeta\left(X_s^{\alpha^*(u_2^\varepsilon)u_2^\varepsilon}\right) \diff s \nonumber \\
        &\quad + \int_{0}^T e^{-\delta s} \mathcal{A}\varphi(X_s^{\alpha^*(u_2^\varepsilon)u_2^\varepsilon})-\delta e^{-\delta s} \varphi(X_s^{\alpha^*(u_2^\varepsilon)u_2^\varepsilon}) \diff s \nonumber\\
        &\quad +  e^{-\delta \tau^{\alpha^*(u_2^\varepsilon)u_2^\varepsilon}} h(X^{\alpha^*(u_2^\varepsilon)u_2^\varepsilon}_{\tau^{\alpha^*(u_2^\varepsilon)u_2^\varepsilon}}) \ind{\{T\geq\tau^{\alpha^*(u_2^\varepsilon)u_2^\varepsilon}\}} \Biggr].
    \end{align}
    Further, note that 
    \begin{equation*}
        \mathbb{P}\left[T_1\leq t \vert T_1^i\leq T_1^j\right]=\mathbb{P}\left[T_1 \leq t\right], \quad i,j \in \{1,2\}.
        \end{equation*}
     Therefore,
    \begin{align*}
        &\mathbb{E}_x \left[e^{-\delta \tau^{\alpha^*(u_2^\varepsilon)u_2^\varepsilon}} h(X^{\alpha^*(u_2^\varepsilon)u_2^\varepsilon}_{\tau^{\alpha^*(u_2^\varepsilon)u_2^\varepsilon}}) \ind{\{T\geq\tau^{\alpha^*(u_2^\varepsilon)u_2^\varepsilon}\}}\right] \\
        &=\mathbb{E}_x \left[ e^{-\delta T_1} h(X^{\alpha^*(u_2^\varepsilon)u_2^\varepsilon}_{T_1}) \ind{\{\tau^{\alpha^*(u_2^\varepsilon)u_2^\varepsilon}=T_1\}} \ind{\{T_1 \leq t^*\}}\right]\\
        &=  \mathbb{E}_x \left[  e^{-\delta T_1} \Biggl(\int_{\rho(X^{\alpha^*(u_2^\varepsilon)u_2^\varepsilon}_{T_1-}-a,\alpha^*(u_2^\varepsilon))}^{\infty} h(X^{\alpha^*(u_2^\varepsilon)u_2^\varepsilon}_{T_1-}-r(y,\alpha^*(u_2^\varepsilon)_{T_1})) \diff F_Y^1(y) \right.\\
        &\quad+ \left. \int_{\rho(b-X^{\alpha^*(u_2^\varepsilon)u_2^\varepsilon}_{T_1-},u_2^\varepsilon)}^{\infty} h(X^{\alpha^*(u_2^\varepsilon)u_2^\varepsilon}_{T_1-}+r(y,u_2^\varepsilon(T_1)) \diff F_Y^2(y) \Biggr)\ind{\{T_1 \leq t^*\}} \right]\\ 
        &= \frac{\lambda_1}{\lambda} \mathbb{E}_x \left[ e^{-\delta T_1} \int_{\rho(X^{\alpha^*(u_2^\varepsilon)u_2^\varepsilon}_{T_1-}-a,\alpha^*(u_2^\varepsilon))}^{\infty} h(X^{\alpha^*(u_2^\varepsilon)u_2^\varepsilon}_{T_1-}-r(y,\alpha^*(u_2^\varepsilon)_{T_1})) \diff F_Y^1(y) \ind{\{T_1 \leq t^*\}} \right] \\
        & \quad +\frac{\lambda_2}{\lambda} \mathbb{E}_x \left[ e^{-\delta T_1} \int_{\rho(b-X^{\alpha^*(u_2^\varepsilon)u_2^\varepsilon}_{T_1-},u_2^\varepsilon)}^{\infty} h(X^{\alpha^*(u_2^\varepsilon)u_2^\varepsilon}_{T_1-}+r(y,u_2^\varepsilon(T_1)) \diff F_Y^2(y) \ind{\{T_1 \leq t^*\}}  \right].
    \end{align*}
    By compensation,
    \begin{align*}
        &\mathbb{E}_x \left[ e^{-\delta T_1} \int_{\rho(X^{\alpha^*(u_2^\varepsilon)u_2^\varepsilon}_{T_1-}-a,\alpha^*(u_2^\varepsilon))}^{\infty} h(X^{\alpha^*(u_2^\varepsilon)u_2^\varepsilon}_{T_1-}-r(y,\alpha^*(u_2^\varepsilon)_{T_1})) \diff F_Y^1(y) \ind{\{T_1 \leq t^*\}} \right] \\
        &=  \mathbb{E}_x \left[ \int_0^{T_1} e^{-\delta t} \int_{\rho(X^{\alpha^*(u_2^\varepsilon)u_2^\varepsilon}_{t-}-a,\alpha^*(u_2^\varepsilon))}^{\infty} h(X^{\alpha^*(u_2^\varepsilon)u_2^\varepsilon}_{t-}-r(y,\alpha^*(u_2^\varepsilon)_{t})) \diff F_Y^1(y) \diff N_t \ind{\{T_1 \leq t^*\}} \right] \\
        &= \mathbb{E}_x \left[ \int_0^{T} e^{-\delta t} \int_{\rho(X^{\alpha^*(u_2^\varepsilon)u_2^\varepsilon}_{t-}-a,\alpha^*(u_2^\varepsilon))}^{\infty} h(X^{\alpha^*(u_2^\varepsilon)u_2^\varepsilon}_{t-}-r(y,\alpha^*(u_2^\varepsilon)_{t})) \diff F_Y^1(y) \diff N_t  \right]\\
        &= \lambda \ \mathbb{E}_x \left[ \int_0^{T} e^{-\delta t} \int_{\rho(X^{\alpha^*(u_2^\varepsilon)u_2^\varepsilon}_{t-}-a,\alpha^*(u_2^\varepsilon))}^{\infty} h(X^{\alpha^*(u_2^\varepsilon)u_2^\varepsilon}_{t-}-r(y,\alpha^*(u_2^\varepsilon)_{t})) \diff F_Y^1(y) \diff t \right].
    \end{align*}
    Similar calculations for the second summand and plugging this into \eqref{eq:Viscosityproof}, leads to
    \begin{align*}
        \varepsilon&\geq \mathbb{E}_x \Biggl[ \int_{0}^{\tau^{\alpha^*(u_2^\varepsilon)u_2^\varepsilon}\wedge T}e^{-\delta s} \zeta\left(X_s^{\alpha^*(u_2^\varepsilon)u_2^\varepsilon}\right) \diff s \nonumber \\
        &\quad + \int_{0}^T e^{-\delta s} \mathcal{A}\varphi(X_s^{\alpha^*(u_2^\varepsilon)u_2^\varepsilon})-\delta e^{-\delta s} \varphi(X_s^{\alpha^*(u_2^\varepsilon)u_2^\varepsilon}) \diff s \nonumber\\
        &\quad +  \lambda_1 \int_0^{T} e^{-\delta t} \int_{\rho(X^{\alpha^*(u_2^\varepsilon)u_2^\varepsilon}_{t-}-a,\alpha^*(u_2^\varepsilon))}^{\infty} h(X^{\alpha^*(u_2^\varepsilon)u_2^\varepsilon}_{t-}-r(y,\alpha^*(u_2^\varepsilon)_{t})) \diff F_Y^1(y) \diff t \\
        &\quad + \lambda_2\int_0^{T} e^{-\delta t} \int_{\rho(b-X^{\alpha^*(u_2^\varepsilon)u_2^\varepsilon}_{t-},u_2^\varepsilon)}^{\infty} h(X^{\alpha^*(u_2^\varepsilon)u_2^\varepsilon}_{t-}+r(y,u_2^\varepsilon(t)) \diff F_Y^2(y) \diff t\Biggr].
    \end{align*}
    \ \\
    On the other hand, the last expectation equals, 
    \begin{align*}
        & \mathbb{E}_x \Biggl[ \int_{0}^{\tau^{\alpha^*(u_2^\varepsilon)u_2^\varepsilon}\wedge T} e^{-\delta s} L(X_s^{\alpha^*(u_2^\varepsilon)u_2^\varepsilon },\alpha^*(u_2^\varepsilon),u_2^\varepsilon) \diff s \Biggr]\\
        & \quad = \mathbb{P}(T_1>t^*) \mathbb{E}_x \Biggl[ \int_{0}^{t^*} e^{-\delta s} L(X_s^{\alpha^*(u_2^\varepsilon)u_2^\varepsilon },\alpha^*(u_2^\varepsilon),u_2^\varepsilon) \diff s \Bigg \vert T_1 >t^*\Biggr] \\
       & \qquad + \mathbb{P}(T_1\leq t^*) \mathbb{E}_x \Biggl[ \int_{0}^{T_1} e^{-\delta s} L(X_s^{\alpha^*(u_2^\varepsilon)u_2^\varepsilon },\alpha^*(u_2^\varepsilon),u_2^\varepsilon) \diff s \Bigg \vert T_1\leq t^*\Biggr]\\
       &\quad \geq \frac{\theta}{2\delta}  \left(e^{-\lambda t^*}(1-e^{-\delta t^*})+\frac{1}{\delta + \lambda}\left(\delta + e^{-(\delta + \lambda)t^*} \lambda - e^{-\lambda t^*} (\delta + \lambda)\right)\right).
    \end{align*}
    If we now choose $\varepsilon$ strictly smaller than the last line, we arrive at a contradiction, therefore \eqref{eq:supersolution} holds. \\
    
To prove that $\overline{V}$ is a subsolution, we have to show that for some smooth function $\psi$ on $[a,b]$ with $\psi(x)=\overline{V}(x)$ and $\overline{V}-\psi \leq 0$,

    \begin{equation}
    \label{eq:subsolution}
        \inf_{u_2 \in A_2}\sup_{u_1 \in A_1}  \mathcal{L}(x,\psi(x),\psi'(x),\psi,u_1,u_2) \geq 0. 
    \end{equation}
    Suppose,there exists some $\eta >0$ such that
    \begin{equation}
    \label{eq:K<0}
        \inf_{u_2 \in A_2}\sup_{u_1 \in A_1}  \mathcal{L}(x,\psi(x),\psi'(x),\psi,u_1,u_2) \leq - \eta <0.
    \end{equation}
    We define 
    \begin{equation*}
        K(x,u_1,u_2):= \mathcal{L}(x,\psi(x),\psi'(x),\psi,u_1,u_2)
    \end{equation*}
    Since $A_2$ is compact, we have by \eqref{eq:K<0} that there exists some $\tilde{u}_2 \in A_2$ such that
    \begin{equation*}
        \sup_{u_1} K(x,u_1,\tilde{u}_2) \leq -\eta.
    \end{equation*}
    Since $K$ is again uniformly continuous in $x$, there exists some $\tilde{h}>0$ such that for $\hat{x}$ with $\vert x-\hat{x}\vert<\tilde{h}$,
    \begin{equation*}
        K(\hat{x},\alpha(\tilde{u}_2),\tilde{u}_2)\leq -\frac{\eta}{2},
    \end{equation*}
    for all $\alpha \in \Gamma$.
    Note that the constant control $\tilde{u}_2$, given by $\tilde{u}^2_t\equiv\tilde{u}_2$ is an admissible control, i.e., $\tilde{u}_2 \in \mathcal{U}_2$. Again, there exists some $\tilde{t}>0$ such that $\sup_{u_1 \in A_1,u_2\in A_2} \vert (c_1(u_1)-c_2(u_2))\vert \tilde{t} \leq \tilde{h}$ and $x \pm \sup_{u_1 \in A_1,u_2\in A_2} \vert (c_1(u_1)-c_2(u_2))\vert \tilde{t} \in (a,b)$. We set $\tilde{T}:=T_1\wedge \tilde{t}$ and $\psi(x)=0$ for $x \notin [a,b]$.
    Let $\varepsilon>0$. Again by Theorem \ref{thm:DPP}, the definition of $\psi$ and the existence of some $\varepsilon$-optimal strategy $\alpha_\varepsilon$, we get by similar calculations as before, that
    
    \begin{align}
        -\varepsilon &\leq \mathbb{E}_x \Biggl[ \int_{0}^{\tau^{\alpha_\varepsilon(\tilde{u}_2)\tilde{u}_2}\wedge \tilde{T}}e^{-\delta s} \zeta\left(X_s^{\alpha_\varepsilon(\tilde{u}_2)\tilde{u}_2}\right) \diff s \nonumber \\
        &\quad + \int_{0}^{\tilde{T}} e^{-\delta s} \mathcal{A}\varphi(X_s^{\alpha_\varepsilon(\tilde{u}_2)\tilde{u}_2})-\delta e^{-\delta s} \varphi(X_s^{\alpha_\varepsilon(\tilde{u}_2)\tilde{u}_2}) \diff s \nonumber\\
        &\quad +  \lambda_1 \int_0^{\tilde{T}} e^{-\delta t} \int_{\rho(X_{t-}-a,\alpha_\varepsilon(\tilde{u}_2))}^{\infty} h(X_{t-}-r(y,\alpha_\varepsilon(\tilde{u}_2)_{t})) \diff F_Y^1(y) \diff t \nonumber\\
        &\quad + \lambda_2\int_0^{\tilde{T}} e^{-\delta t} \int_{\rho(b-X_{t-},\tilde{u}_2)}^{\infty} h(X_{t-}+r(y,\tilde{u}_2) \diff F_Y^2(y) \diff t\Biggr].
    \end{align}
As before the last expectation is,
 \begin{align*}
     &\mathbb{E}_x \Biggl[ \int_{0}^{\tau^{\alpha_\varepsilon(\tilde{u}_2)\tilde{u}_2}\wedge \tilde{T}} e^{-\delta s} K(X_s^{\alpha_\varepsilon(\tilde{u}_2)\tilde{u}_2 },\alpha_\varepsilon(\tilde{u}_2),\tilde{u}_2) \diff s \Biggr]\\
     & \quad \leq  -\frac{\eta}{2\delta}  \left(e^{-\lambda \tilde{t}}(1-e^{-\delta \tilde{t}})+\frac{1}{\delta + \lambda}\left(\delta + e^{-(\delta + \lambda)\tilde{t}} \lambda - e^{-\lambda \tilde{t}} (\delta + \lambda)\right)\right).
 \end{align*}
  By choosing $\varepsilon < \frac{\eta}{2\delta}  \left(e^{-\lambda \tilde{t}}(1-e^{-\delta \tilde{t}})+\frac{1}{\delta + \lambda}\left(\delta + e^{-(\delta + \lambda)\tilde{t}} \lambda - e^{-\lambda \tilde{t}} (\delta + \lambda)\right)\right) $  we obtain another contradiction.
\end{proof}

\subsection{Uniqueness of solution to Bellman-Isaacs equation}

Since $\underline{V}$ and $\overline{V}$ fulfill \eqref{DPPlower} and \eqref{DPPupper}, it seems reasonable to use a fixed point argument to show uniqueness of the solution. Unfortunately, a direct link between a fixed point and (viscosity-)solutions can only be achieved if we have absolute continuity and the existence of Markovian optimal controls. Therefore, we are going to show a classical comparison theorem for the set of Bellman-Isaacs equations. The proof follows an idea of \cite{AzcueMuler2014}.
\begin{theorem}
Let $\underline{v}:[a,b]\to\mathbb{R}$ and $\overline{v}:[a,b]\to\mathbb{R}$ be continuous and bounded viscosity sub- and supersolutions to \eqref{eq:BI_upper} (or \eqref{eq:BI_lower}). If it holds that $\underline{v}(a)\leq\overline{v}(a)$ and $\underline{v}(b)\leq\overline{v}(b)$, then $\underline{v}(x)\leq \overline{v}(x)$ for all $x\in[a,b]$.
\end{theorem}

\begin{proof}
    Assume that there exists some $x_0 \in (a,b)$ where $\underline{v}(x_0)>\overline{v}(x_0)$. Then,
    \begin{equation}
        0< \max_{x\in [a,b]} \left(\underline{v}(x)-\overline{v}(x)\right) = \widetilde{M},
        \label{eq:Max>0}
    \end{equation}
    with maximizing argument $x^*$. We define for $\nu>0$ and $(x,y)\in[a,b]^2$,
    \begin{equation*}
        G_\nu(x,y) := \underline{v}(x)-\overline{v}(x)-g^\nu(x,y),
    \end{equation*}
    where 
    \begin{equation*}
        g^\nu(x,y) := \nu (x-y)^2,
    \end{equation*}
    and set
    \begin{equation*}
        M_\nu = \max_{(x,y)\in [a,b]^2} G_\nu (x,y),
    \end{equation*}
   again with with maximizer $(x_\nu,y_\nu)$. We see, that
    \begin{equation*}
        M_\nu \geq G_v(x^*,x^*) = \widetilde{M},
    \end{equation*}
    and therefore,
   \begin{equation}
    \label{eq:M_nu>0}
       \liminf_{\nu \rightarrow \infty} M_{\nu} \geq \widetilde{M} >0.
    \end{equation}
    To exploit the differentiability of $g^\nu$, we first need to show that $(x_\nu,y_\nu)$ is not an element of the boundary of $[a,b]^2$, i.e., $(x_\nu,y_\nu) \in (a,b)^2$.
    Indeed, we immediately see
    \begin{equation*}
        G_\nu (a,a) \leq 0, \quad G_\nu(b,b) \leq 0,
    \end{equation*}
    but we know that $M_\nu >0$ from \eqref{eq:M_nu>0}. Furthermore, for $(x,y) \in \partial[a,b]^2$ with $|x-y|\geq\sqrt{\varepsilon}>0$ - the edges of the rectangle and the vertices $\{(a,b),\,(b,a)\}$,
    \begin{equation*}
        G_\nu(x,y)= \underline{v}(x)-\overline{v}(y)-\nu (x-y)^2 \leq 2M -\nu \varepsilon <0,
    \end{equation*}
    for $\nu> 2M/\varepsilon$, where $M$ is a bound of both $\overline{v}$ and $\underline{v}$. Therefore, $(x_\nu,y_\nu) \notin \partial[a,b]^2$ for $\nu$ large enough.\\
    Since $G_\nu$ attains a maximum at $(x_\nu,y_\nu)$, we have that for any $x \in [a,b]$, $G_\nu(x,y_\nu) \leq G_\nu(x_\nu,y_\nu) $ which translates to 
    \begin{equation*}
        \underline{v}(x)-\underline{v}(x_\nu) \leq g^\nu (x,y_\nu) - g^\nu (x_\nu,y_\nu).
    \end{equation*}
    Taylor's theorem then yields
    \begin{equation*}
        \limsup_{x\rightarrow x_\nu} \frac{\underline{v}(x)-\underline{v}(x_\nu)-g_x^\nu(x_\nu,y_\nu)(x-x_\nu)-o(|x-x_\nu|)}{|x-x_\nu|} \leq 0.
    \end{equation*}
    This implies that $g_x^\nu(x_\nu,y_\nu)\in D^+(\underline{v})(x_\nu)$, which is the set of all super-differentials of $\underline{v}$ in $x_\nu$.\\
    Equivalently, for any $y \in [a,b]$ we have $ G_\nu(x_\nu,y) \leq G_\nu(x_\nu,y_\nu)$ and
    \begin{equation*}
         \overline{v}(y)-\overline{v}(y_\nu) \geq g^\nu (x_\nu,y_\nu) - g^\nu (x_\nu,y).
    \end{equation*}
    Again this leads to 
    \begin{equation*}
        \liminf_{y\rightarrow y_\nu}\frac{\overline{v}(y)-\overline{v}(y_\nu) - \left(-g_y^\nu(x_\nu,y_\nu)(y-y_\nu)\right)+o(|y-y_\nu|)}{|y-y_\nu|}\geq 0.
    \end{equation*}
    Thus, $-g_y^\nu(x_\nu,y_\nu) \in D^-(\overline{v})(y_\nu)$, which is the set of all sub-differentials of $\overline{v}$ in $y_\nu$. 
    Since $\underline{v}$ and $\overline{v}$ are sub- and supersolution of $\eqref{eq:BI_upper}$, the following holds:
    \begin{align}
        \inf_{u_2 \in A_2} \sup_{u_1 \in A_1} \mathcal{L}(x_\nu,\underline{v}(x_\nu),g_x^\nu(x_\nu,y_\nu),\underline{v}(x_\nu),u_1,u_2)\geq 0, \label{eq:proof_line1}\\
        \inf_{u_2 \in A_2} \sup_{u_1 \in A_1} \mathcal{L}(y_\nu,\overline{v}(y_\nu),-g_y^\nu(x_\nu,y_\nu),\overline{v}(y_\nu),u_1,u_2)\leq 0. \label{eq:proof_line2}
    \end{align}
    By compactness, there exist $\underline{u}_1,\underline{u}_2$ and $\overline{u}_1,\overline{u}_2$  corresponding to optimizers in \eqref{eq:proof_line1} and \eqref{eq:proof_line2} respectively. Additionally, note that $g_x^\nu(x_\nu,y_\nu)=-g_y^\nu(x_\nu,y_\nu)$.\\
    Now, we subtract \eqref{eq:proof_line1} from \eqref{eq:proof_line2}, choose $u_1=\underline{u}_1$ and $u_2 = \overline{u}_2$ and get
        \begin{align}
        &(\delta + \lambda_1+\lambda_2)\left(\underline{v}(x_\nu)-\overline{v}(y_\nu)\right)\leq \zeta(x_\nu)-\zeta(y_\nu) \nonumber\\
        &+\lambda_1 \left(\int_{0}^{\rho(x_\nu-a,\underline{u}_1)}\underline{v}(x_\nu-r(y,\underline{u}_1))\diff F_Y^1(y)- \int_{0}^{\rho(y_\nu-a,\underline{u}_1)}\overline{v}(y_\nu-r(y,\underline{u}_1))\diff F_Y^1(y)\right) \nonumber\\
        & + \lambda_2 \left(  \int_{0}^{\rho(b-x_\nu,\overline{u}_2)} \underline{v}(x_\nu + r(y,\overline{u}_2)) \diff F_Y^2(y)- \int_{0}^{\rho(b-y_\nu,\overline{u}_2)} \overline{v}(y_\nu + r(y,\overline{u}_2)) \diff F_Y^2(y)\right) \nonumber\\
        &+\lambda_1 \left( \int_{\rho(x_\nu-a,\underline{u}_1)}^{\infty}h(x_\nu-r(y,\underline{u}_1))\diff F_Y^1(y) - \int_{\rho(y_\nu-a,\underline{u}_1)}^{\infty}h(y_\nu-r(y,\underline{u}_1))\diff F_Y^1(y) \right) \nonumber \\
        &+ \lambda_2 \left( \int_{\rho(b-x_\nu,\overline{u}_2)}^{\infty}h(x_\nu+r(y,\overline{u}_2))\diff F_Y^2(y) - \int_{\rho(b-y_\nu,\overline{u}_2)}^{\infty}h(y_\nu+r(y,\overline{u}_2))\diff F_Y^2(y) \right). \label{eq:proof_inequality}
    \end{align}
    To carry on, we need further properties of the maximizer $(x_\nu,y_\nu)$, especially for $\nu \rightarrow \infty$. At first we observe that
    \begin{equation*}
        G_\nu(x_\nu,x_\nu)+G_\nu(y_\nu,y_\nu)\leq 2 G_\nu(x_\nu,y_\nu),
    \end{equation*}
    which gives us the following estimate on $(x_\nu - y_\nu)^2$:
    \begin{align}
        \underline{v}(x_\nu)-\overline{v}(x_\nu)+\underline{v}(y_\nu)-\overline{v}(y_\nu)\leq 2 \left(\underline{u}(x_\nu)-\overline{v}(y_\nu)-\lambda(x_\nu - y_\nu)^2\right) \nonumber\\
        \Leftrightarrow 2 \nu (x_\nu-y_\nu)^2 \leq \underline{v}(x_\nu)-\overline{v}(y_\nu)\leq 2 M \nonumber\\
        \Leftrightarrow (x_\nu - y_\nu)^2\leq \frac{M}{\nu}. \label{eq:x=y}
    \end{align}
    Using the compactness of $[a,b]^2$, we have existence of a sequence $(\nu_n)_{n \in \mathbb{N}}$ with $\nu_n \rightarrow \infty$, such that $\lim_{n\to\infty}(x_{\nu_{n}},y_{\nu_{n}})=(\overline{x},\overline{y}) \in [a,b]^2$.\\
    From \eqref{eq:x=y} we get $\overline{x}=\overline{y}$, together with \eqref{eq:proof_inequality} and \eqref{eq:Max>0}, this gives us
    \begin{align*}
         &(\delta + \lambda_1+\lambda_2)\left(\underline{v}(\overline{x})-\overline{v}(\overline{x}))\right) \\
         & \quad \leq 
         \lambda_1 \int_{0}^{\rho(\overline{x}-a,\underline{u}_1)}\underline{v}(\overline{x}-r(y,\underline{u}_1))-\overline{v}(\overline{x}-r(y,\underline{u}_1))\diff F_Y^1(y)\\
         & \qquad + \lambda_2 \int_{0}^{\rho(b-\overline{x},\overline{u}_2)} \underline{v}(\overline{x}+ r(y,\overline{u}_2))- \overline{v}(\overline{x}+ r(y,\overline{u}_2))\diff F_Y^2(y) \leq (\lambda_1+\lambda_2) \widetilde{M}
    \end{align*}
    Altogether, we arrive at the contradiction
    \begin{equation*}
        \widetilde{M}\leq \liminf_{\nu \rightarrow \infty} M_\nu \leq \lim_{n \to \infty} M_{\nu_n} = \underline{v}(\overline{x})-\overline{v}(\overline{x}) \leq \frac{\lambda_1+\lambda_2}{\delta+\lambda_1+\lambda_2} \widetilde{M}.
    \end{equation*}
   This proves the statement of the theorem.
\end{proof}

\begin{remark}
Since the involved operators are non-local, the boundary values in $a$ and $b$ are not necessarily a-priori known, indeed they are part of the solution. Intuitively, this follows from controlling the distribution of the post-jump location and the sign of the drift close to the boundaries. $\underline{V}$ and $\overline{V}$ are both sub- and supersolutions to the respective equations, and therefore their uniqueness - as solutions to the Bellman-Isaacs equations - follows from the comparison theorem and the implicitly given boundary values.
\end{remark}

\section{Numerical examples}
In the previous sections, we discussed that the value functions are viscosity solutions to Bellman-Isaacs equations. From the theory of viscosity solutions we know that if the functions $\underline{V}$ and $\overline{V}$ were continuously differentiable, they would be true solutions to these equations. Therefore, we are now trying to find numerical solutions and corresponding controls. Note that if we find controls $u_1^*$ and $u_2^*$ for which the upper and lower value coincide, a Nash equilibrium is immediately achieved. 

For the numerical approach, we use a policy iteration algorithm:  We divide the interval $[a,b]$ into an equidistant grid $\{x_1,\dots,x_n\}$, where $a=x_1\leq x_2\leq\dots\leq x_n=b$ and set $u_1^0=u_2^0=\overline{u}$ where $\overline{u}$ corresponds to the constant control of no reinsurance. Then, we calculate $J^{u_1^0 u_2^0}$ by numerically solving  
\begin{equation}
    \mathcal{L}(x,J^{u_1^0 u_2^0}(x),(J^{u_1^0 u_2^0})^{'}(x),J^{u_1^0 u_2^0},u_1^0,u_2^0)=0, \label{eq:numSolve}
\end{equation}
for every $x \in \{x_2,\dots,x_{n-1}\}$. On the boundary, i.e., for $x_1=a$ and $x_n=b$, we might have to either solve an equation or deal with a boundary value. We will discuss this later in a concrete example. Then, we update the controls one at a time, for example when starting with updating the control of player 2, we set for $x \in \{x_1,\dots,x_n\}$,
\begin{equation*}
    u_2^1(x) = \argmin_{u \in A_2}  \mathcal{L}(x,J^{u_1^0 u_2^0}(x),(J^{u_1^0 u_2^0})^{'}(x),J^{u_1^0 u_2^0},u_1^0,u).
\end{equation*}
Afterwards, we calculate $J^{u_1^0 u_2^1}$ like in \eqref{eq:numSolve} but with $u_2^1$ instead of $u_2^0$. Then, we find $u_1^1$ by 
\begin{equation*}
    u_1^1(x) = \argmax_{u \in A_1}  \mathcal{L}(x,J^{u_1^0 u_2^1}(x),(J^{u_1^0 u_2^1})^{'}(x),J^{u_1^0 u_2^1},u,u_2^1),
\end{equation*}
$x \in \{x_1,\dots,x_n\}$ and again solve the equation for $J^{u_1^1 u_2^1}$. We proceed in this way for $u_2^2,u_1^2,\dots$ until there are no more significant changes in the controls or value function. I.e., we fix a level $\varepsilon\approx 10^{-5}$ and stop the iteration if the maximal improvement on the grid is smaller than $\varepsilon$. Markov controls are inherently used by the nature of the schema. This procedure is also motivated by \cite{FlemingSoner2006}, p.~379.

We consider the following (test-) functional:
\begin{equation}
    J^{u_1 u_2}(x)=\mathbb{E}_x \left[e^{-\delta \tau^{u_1 u_2}} \ind{\{X^{u_1 u_2}_{\tau^{u_1 u_2}} \geq b \}} \right],
\end{equation}
with corresponding value functions $\underline{V}$ and $\overline{V}$. This results from setting $\zeta(x)=0$ and $h(x)=\ind{\{x\geq b\}}$. If $\delta = 0$, this is just the probability that the process exits the interval $[a,b]$ at the upper endpoint. It is intuitive that player one would aim to maximize this probability, as a larger probability corresponds to a greater expected surplus relative to player two. Conversely, player two wants to prevent the process from outgrowing $b$, therefore tries to find a minimizing strategy. Note that in this section, the terms ``control" and ``strategy" are used interchangeably.

\subsection{Proportional reinsurance, variance principle and exponential claims}
\label{sec:proportional}
\begin{figure}[h]
		\includegraphics[width=1.0\textwidth]{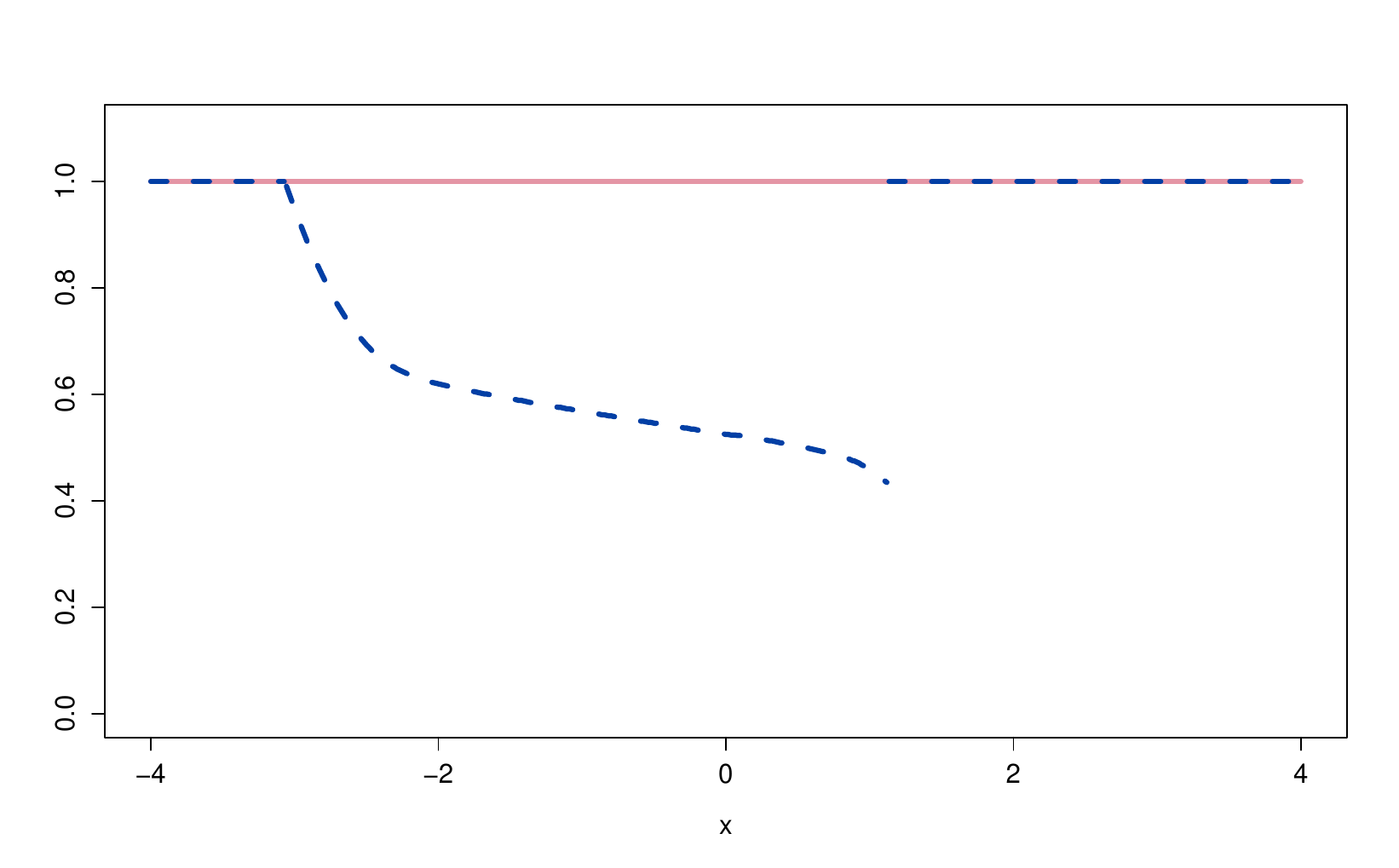}
		\caption{Strategies of player 1 and 2 after 10 iterations. The controls of Player 1 and 2 correspond to the solid and dashed line respectively.}
		\label{fig:u1andu2_prop}
\end{figure}
We start by setting $a=-4$, $b=4$ and $A=A_1=A_2=[0,1]$. The risk share function $r:\mathbb{R}_0^+ \times A \rightarrow \mathbb{R}_0^+$ corresponding to proportional reinsurance is given by $r(y,u)=yu$. We assume that the claim heights are exponentially distributed with rate $1/\mu_i$, for $i \in \{1,2\}$. For the first insurer's premium without acquiring reinsurance we apply the expectation principle, i.e., for $i \in \{1,2\}$, $c_i = \lambda_i (1+\eta_i) \mu_i$, with risk loading $\eta_i>0$ and $\mu_i= \mathbb{E}(Y_1^i)$. For the reinsurer's premium, we apply the variance principle which leads us to the following premium function:
\begin{equation*}
    c_i(u)=c_i-\lambda_i(\mu_i(1-u)+\theta_i\mu_i^2(1-u)^2),
\end{equation*}
where $\theta_i>0$ is the safety loading corresponding to the reinsurance premium.
In Figure~\ref{fig:u1andu2_prop}, we used the parameters $\lambda_1=\lambda_2=1$, $\delta= 0.05$, $\mu_1=1$, $\mu_2=2$, $\eta_1=0.1$, $\eta_2=0.08$, $\theta_1=\theta_2=0.12$. Here, we see that $c_2=c_2(1)>c_1(u)$, for any $0\leq u \leq 1$. Therefore, close to $a$ player 2 will take the chance and produce a maximal negative drift, which results in the boundary value $\underline{V}(a)=\overline{V}(a)=0$. In the middle of the interval, reinsurance is actively purchased, whereas close to $b$ no reinsurance is taken in order to delay an upper exit. For player one, it is better not to acquire any reinsurance but to maximize the drift in order to exceed the upper bound of the interval as fast as possible. This also means that this player voluntarily accepts the full extent of potential claims.

\subsection{Excess of loss reinsurance and expectation principle}

In this framework, we set $a=-2$ and $b=2$, as well as $A=A_1=A_2=[0,\infty]$. The risk share function for excess of loss reinsurance is given by $r(y,M)=\min(y,M)$. The claim amount which exceeds $M$ is covered by the reinsurer. The first insurers premium is calculated like in Section \ref{sec:proportional}.

\begin{figure}[ht!]
	\includegraphics[width=1\textwidth]{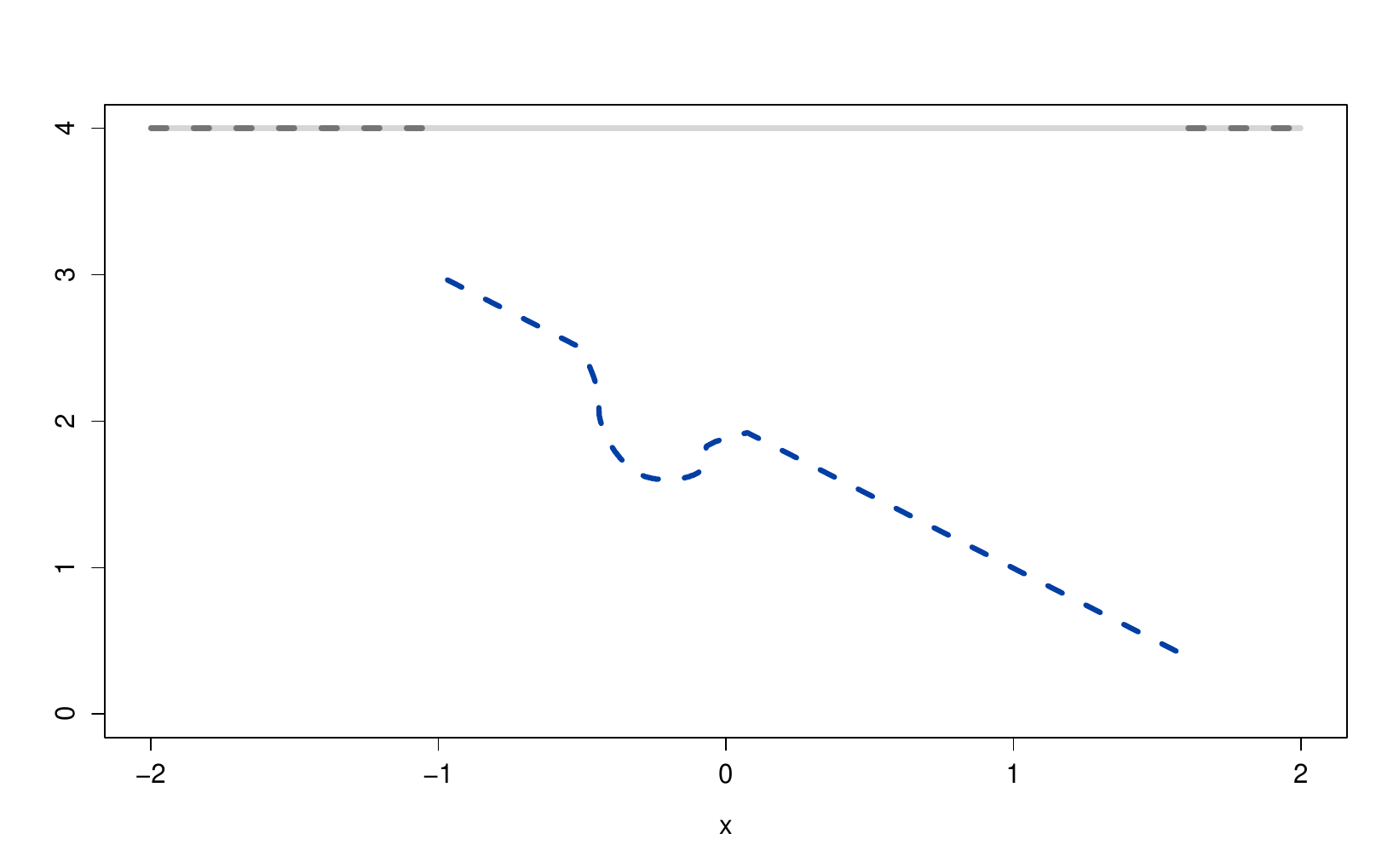}
	\caption{Strategies of player 1 and 2 after 10 iterations with exponentially distributed claims. The solid line represents player one and the dashed player two. Technically, ``no reinsurance" corresponds to an infinite value of the control, but for illustrative purposes it is shown in  gray.}
	\label{fig:u1andu2_EoL_exponential}
	\includegraphics[width=1\textwidth]{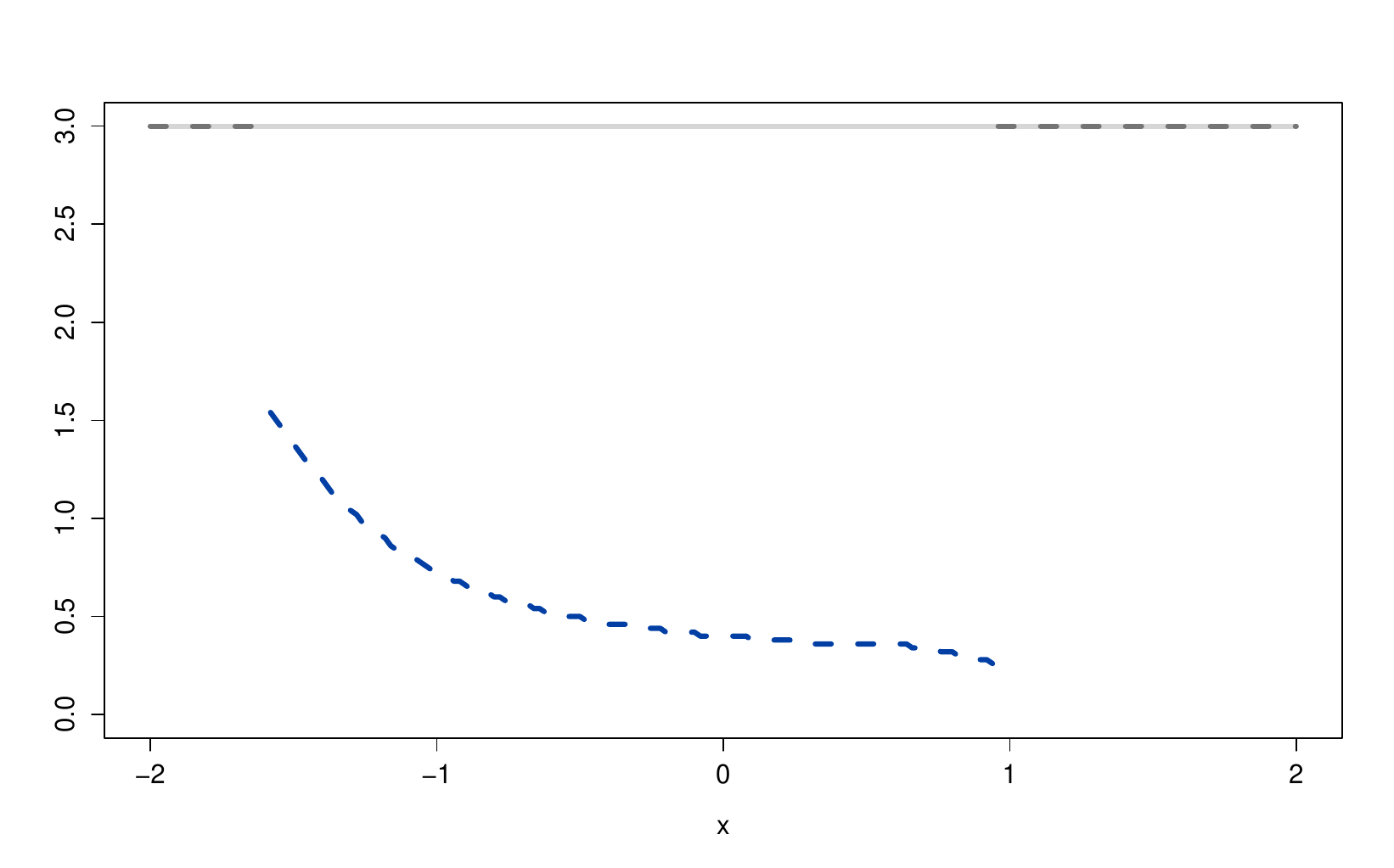}
	\caption{Strategies of player 1 and 2 after 10 iterations with Pareto II distributed claims. Again, the solid line shows the strategy of player one and the dashed one the strategy of player two, and the color gray stands for ``no reinsurance".}
	\label{fig:u1andu2_EoL_Pareto}
\end{figure}

\subsubsection{Exponential claims}
Let the claims be exponentially distributed with rate $1/\mu_i$, for $i \in \{1,2\}$. The premium function is now calculated via the expectation principle, i.e.,
\begin{equation*}
    c_i(M) = c_i-\lambda_i (1+\theta_i)\mu_i e^{-M/\mu_i}.
\end{equation*}
With the same parameter values as in Section \ref{sec:proportional}, we arrive at Figure \ref{fig:u1andu2_EoL_exponential}. The plot suggests that in a certain range, the optimal control of player 2 is linear in $x$. Perhaps certain adjustments in the simulation procedure would lead to a straight line. However, in the interest of comparability and computational effort, the method remained unchanged across all examples - same discretization and termination criterion.

\subsubsection{Pareto type II distributed claims}
 Now, we assume that the claims are Pareto II distributed with shape $\alpha_i$ and scale $\beta_i$. By applying the expectation principle, the premium function is given by
\begin{equation*}
    c_i(M) = c_i-\lambda_i (1+\theta_i) \left(\frac{\beta_{i}^{\alpha_i}}{\alpha_i-1}(M+\beta_i)^{-\alpha_i+1} \right).
\end{equation*}
In Figure \ref{fig:u1andu2_EoL_Pareto}, we chose $\alpha_1=\alpha_2=3$, $\beta_1=\beta_2=1$ and retain the remaining ones from Section \ref{sec:proportional}.
For these parameters, $\lim_{m \to \infty} c_1(m) > \lim_{m \to \infty} c_2(m)$, therefore  $\underline{V}(b)=\overline{V}(b)=1$.

\section*{Declarations}

This research was funded in whole or in part by the Austrian Science Fund (FWF) [10.55776/P33317]. For open access purposes, the authors have applied a CC BY public copyright license to any author-accepted manuscript version arising from this submission.

\bibliographystyle{plainnat}

\end{document}